\documentclass[11 pt,fullpage]{article}
\usepackage{amsfonts,amssymb}
\usepackage{ graphicx, mathrsfs}
\usepackage{amsmath,amsthm}
\usepackage{color} 
\usepackage[dvipsnames]{xcolor}
\usepackage{hyperref}
\usepackage[right = 2.5cm, left=2.5cm, top = 2.5cm, bottom =2.5cm]{geometry}
\usepackage{comment}
\usepackage{mathtools}
\usepackage{graphicx}
\graphicspath{ {./images/} }

\mathtoolsset{showonlyrefs=true}

\theoremstyle{plain}
\newtheorem{theorem}{Theorem}

\newtheorem{proposition}{Proposition}[section]
\newtheorem{lemma}[proposition]{Lemma}

\theoremstyle{definition}
\newtheorem{remark}{Remark}

\allowdisplaybreaks

\numberwithin{equation}{section}

\newcommand\R{{\mathbb R}}
\newcommand\T{{\mathbb T}}

\newcommand{\TT}{\mathbb{T}}

\renewcommand{\Re}{\textup{Re}\,}
\renewcommand{\Im}{\textup{Im}\,}

\newcommand{\cE}{\mathcal E}

\newcommand{\cK}{\mathcal K}
\newcommand{\cL}{\mathcal L}

\newcommand{\cQ}{\mathcal Q}

\def\eps{{\varepsilon}}

\renewcommand\eps{\epsilon }

\newcommand{\Real}{\mathbb R}
\newcommand{\Complex}{\mathbb C}
\newcommand{\Integer}{\mathbb Z}

\newcommand{\norm}[1]{\left\lVert#1\right\rVert}
\newcommand{\abs}[1]{\left\vert#1\right\vert}
\newcommand{\set}[1]{\left\{#1\right\}}

\newcommand{\grad}{\nabla}

\newcommand{\Naturals}{\mathbb N}
\newcommand{\Integers}{\mathbb Z}

\newcommand{\ddv}{\, dv}

\newcommand{\dss}{\displaystyle}

\newcommand{\brak}[1]{\langle{#1}\rangle}

\newcommand{\n}[1]{{\left\| #1 \right\|}}

\newcommand{\rom}[1]{\textup{\uppercase\expandafter{\romannumeral#1}}}

\def\nn{\nonumber}
\def\aa{\alpha}
\def\bb{\beta}

\begin{document}
 
\title{The linearized Vlasov and Vlasov-Fokker-Planck equations in a uniform magnetic field} 
\author{Jacob Bedrossian\footnote{\textit{jacob@cscamm.umd.edu}, University of Maryland, College Park. The author was partially supported by NSF CAREER grant DMS-1552826. Additionally, the research was supported in part by NSF RNMS \#1107444 (Ki-Net).} and Fei Wang\footnote{\textit{fwang256@umd.edu}, University of Maryland, College Park.}} 
\date{\today}
\maketitle

\begin{abstract}
We study the linearized Vlasov equations and the linearized Vlasov-Fokker-Planck equations in the weakly collisional limit in a uniform magnetic field. 
In both cases, we consider periodic confinement and Maxwellian (or close to Maxwellian) backgrounds.  
In the collisionless case, for modes transverse to the magnetic field, we provide a precise decomposition into a countably infinite family of standing waves for each spatial mode. These are known as Bernstein modes in the physics literature, though the decomposition is not an obvious consequence of any existing arguments that we are aware of. 
We show that other modes undergo Landau damping. 
In the presence of collisions with collision frequency $\nu \ll 1$, we show that these modes undergo uniform-in-$\nu$ Landau damping and enhanced collisional relaxation at the time-scale $O(\nu^{-1/3})$. The modes transverse to the field are uniformly stable and exponentially thermalize on the time-scale $O(\nu^{-1})$. 
Most of the results are proved using Laplace transform analysis of the associated Volterra equations, whereas a simple case of Yan Guo's energy method for hypocoercivity of collision operators is applied for stability in the collisional case. 
\end{abstract}


\setcounter{tocdepth}{2}
{\small\tableofcontents}

\section{Introduction}
In this paper we consider the linearized Vlasov and Vlasov-Fokker-Planck equations with a constant background magnetic field in a periodic box. 
We consider only the single-species case, though the analysis should extend (to some degree) to cover multi-species problems in the collisionless case. 
We are specifically interested  in characterizing the Landau damping, the non-damping modes (known as \emph{Bernstein modes} in the physics literature), and the enhanced collisional relaxation in the limit $\nu \rightarrow 0$.   
The unknown perturbation $h: \Real_+ \times \T^3 \times \Real^3 \rightarrow \Real_+$ is taken to be charge zero $\int_{\T^3 \times \Real^3} h(t,x,v) dx dv = 0$ and satisfies 
\begin{equation}\label{def:VPElin}
\left\{
\begin{array}{l} \dss 
\partial_t h + v\cdot \grad_x h + \frac{q}{m}v \times B_0 \cdot \grad_v h + \frac{q}{m} E(t,x) \cdot \grad_v f^0  = \nu (\Delta_v h + \grad_v \cdot (vh )), \\
E(t,x) = -\grad_x W \ast_{x} \rho(t,x), \\ 
\dss \rho(t,x) = \int_{\R^d} h  (t,x,v) dv, \\ 
h(t=0,x,v) = h_{in}(x,v),
\end{array}
\right.
\end{equation}
where 
\begin{align}
\widehat{W}(k) = \frac{q}{4 \pi \abs{k}^2},  \label{def:Coulomb}
\end{align}
 $q>0$ is the charge (chosen positive for the sake of simplicity), $m>0$ is the mass, and $\nu \geq 0$ is the collision frequency. 
We will assume that the magnetic field is given by $B_0 = (0,0,b)$ for a fixed number $b> 0$.
We are only interested  in the case when $\nu = 0$ (collisionless) or $\nu \rightarrow 0$ (weak collision limit).  
For future reference, define $v_{\perp} = (v_1,v_2,0)$ the projection to the directions transverse to the magnetic field (similarly $k_{\perp}$). 
In what follows define the \emph{cyclotron frequency} 
\begin{align}
\omega_c = \frac{qb}{m}, 
\end{align}
which is the rate at which charged particles in a magnetic field gyrate around field lines; see \S\ref{sec:pte}.  
Frequencies of the form $n\omega_c$ with $n \in \Integers_\ast$ are called \emph{cyclotron harmonics}. 

In 1946, Landau \cite{Landau46} observed in the case that $\nu = 0$ and $b=0$, the linearized Vlasov equation predict a rapid decay of the electric field (if $x \in \T^d$) 
despite the lack of any dissipative mechanism. 
In particular, for a large class of smooth equilibria $f^0$ (identified by Penrose \cite{Penrose}) one can prove the following velocity-averaging-type estimates for all $\sigma \geq 0$ and $m > (d - 1)/2$ (for $(x,v) \in \T^d \times \Real^d$) 
\begin{subequations}\label{ineq:velavg}
\begin{align}
\norm{\abs{\grad_x}^{1/2} \brak{\grad_x, \grad_x t}^{\sigma} \rho }_{L^2_t L^2_x} & \lesssim_{\sigma,f^0} \norm{h_{in}}_{H^\sigma_m} \\  
\norm{\abs{\grad_x}^{1/2} e^{\lambda \brak{\grad_x, \grad_x t}} \rho }_{L^2_t L^2_x} & \lesssim_{\lambda,f^0} \norm{e^{\lambda \brak{\grad}} h_{in}}_{L^2_m}. 
\end{align}	
\end{subequations}
The latter estimate requires analyticity of $f^0$ and only holds for $\lambda$ sufficiently small. 
This rapid decay of the density fluctuation is known as \emph{Landau damping}.  
It was observed experimentally in \cite{MalmbergWharton64,MalmbergWharton68} and is now considered one of the most fundamental properties of collisionless plasmas (see e.g. \cite{Ryutov99,GoldstonRutherford95,Stix,Swanson}). 
See \cite{Degond86,Glassey94,Glassey95,MouhotVillani11} for more modern mathematical treatments of linearized Landau damping.
If viewed in $(x,v) \in \T^d_x \times \Real^d_v $, the evolution of the distribution function during Landau damping resembles a passive scalar being stirred by a shear flow. This causes the homogenization in $x$ and drives the associated decay of velocity averages such as the density. This, and similar effects, are often referred to as \emph{phase mixing}.  
For the free transport equation $\partial_t h + v \cdot \grad_x h = 0$, \eqref{ineq:velavg} follows from the Fourier transform and the Sobolev trace lemma (see e.g. Lemma \ref{lem:passive}). To see this for the linearized Vlasov equations with $b = 0$ and $\nu = 0$ one takes advantage of the special structure that allows to reduce the problem to a scalar Volterra equation for $\hat{\rho}(t,k)$. The Volterra equation is then analyzed via the Laplace transform. In the case $b=0$ and $\nu = 0$, a variety of nonlinear results have also been obtained \cite{CagliotiMaffei98,HwangVelazquez09,MouhotVillani11,BMM13,B16} to investigate in which cases the estimates \eqref{ineq:velavg} hold also for nonlinear solutions.   

Nearly all plasmas of physical interest are subject to significant external magnetic fields and it has profound effects on the dynamics \cite{GoldstonRutherford95,Stix,Swanson}. 
Hence, physicists immediately recognized the need to extend the work of Landau to the case $b \neq 0$; this was done by Bernstein in \cite{Bernstein58}. 
Bernstein found that spatial Fourier modes of the density that did not depend on $z$, e.g.  $k = k_{\perp}$, were not subject to Landau damping. Instead, he found an infinite family of standing waves for each mode $k$, one for each harmonic of the cyclotron frequency. These are now known as \emph{Bernstein modes}. 
On the other hand, he also predicted that modes $k$ with $k_3 \neq 0$ are still damped. 

Exactly how weak collisions interact with phase mixing was studied already in the 1950's and had been the subject of some debate in the physics community   (see e.g. \cite{LenardBernstein1958,SuOberman1968,ONeil1968,Johnston1971,NgBhattacharjee1999,NgBhattacharjee2006,ShortSimon2002,Callen2014} for discussions). 
However, a mathematically rigorous study of the case $1 \gg \nu > 0$ and $b = 0$ was completed only recently in \cite{Tristani2016,B17}. 
The work of \cite{B17} confirmed that the phase mixing enhances the effect of collisions, as predicted formally by some physicists \cite{SuOberman1968}, in both the linearized and nonlinear problems.  
Specifically,  the time-scale for collisional relaxtion of $x$-dependent modes in this case is $O(\nu^{-1/3})$. 
The phase mixing creates oscillations in the velocity variable which then in turn enhances the effect of the second order $\Delta_v$ in the collision operator; see Remark \ref{rmk:KolEqn}.

In this work, we provide a detailed study of the linearized dynamics in the case $b \neq 0$ for both collisionless and weak collision limits  $1 \gg \nu \geq 0$. 
The linearized problem is significantly more difficult than the $b = 0$ case studied previously, as the dynamics are significantly more complicated. 
We will first state the results and then provide a discussion. 
In the collisionless case $\nu = 0$, we will study the following class of equilibria: 
\begin{align}
f^0(v) & = \frac{1}{(2\pi)}e^{-\frac{\abs{v_{\perp}}^2}{2}} f_3^0(v) \\ 
f_3^0(v) & = \frac{1}{(2\pi T_{||})^{1/2}} e^{-\frac{v_3^2}{2 T_{||}}}  + \tilde{f}^3(v_3^2),
\end{align}
for  $\abs{T_{||} - 1}$ sufficiently small and $\tilde{f}^3$ in $H^s_m$ for suitable $s,m$ (see below for definitions). 
For simplicity we have normalized the transverse temperature to one. 
This particular class of equilibria is quite natural from a physics perspective, as plasmas are often observed to have different temperatures transverse and parallel to the magnetic field; see e.g. \cite{Stix,Swanson}.  

We prove the following theorem in the case $\nu = 0$, $b \neq 0$. 
Aside from providing the Landau damping estimate \eqref{ineq:LDdmpThm}, Theorem \ref{thm:Cless} provides the decomposition \ref{eq:BernExp}, a countably infinite number of standing waves for each spatial mode. 
This decomposition does not immediately follow from Bernstein's work \cite{Bernstein58}, not even formally\footnote{The quantity in equation (22) of \cite{Bernstein58} is never an entire function. Regardless of the regularity or velocity localization of the initial condition, there are generically poles at every cyclotron harmonic. See \S\ref{sec:pte} for more details.}.  The decomposition is possible due to the singularities in the Laplace transform known as resonances; see Remark \ref{rmk:Resonance} below for more details.
Deriving the formula requires understanding interplay between the Laplace transform description of the linearized Vlasov and the passive transport equation (the linearized Vlasov equation for the case $f^0 \equiv 0$). 

\begin{theorem}[Collisionless dynamics] \label{thm:Cless}
	Let $\brak{v}^m h_{in} \in H^\sigma$ for $\sigma \geq 0$ and $m > 2$. 
	Suppose that $\norm{\tilde{f}^3}_{H^{\sigma'}_m} \leq \delta_0$ with $\sigma' > \sigma + 5/2$ and $\abs{T_{||} - 1} + \delta_0$ sufficiently small depending on universal constants and $b$. 
	Then the following holds:  
	\begin{itemize} 
		\item the Landau damping of $z$-dependent modes:  
		\begin{align}
		\norm{\abs{\partial_z}^{1/2}\brak{\grad,\partial_z t}^\sigma \rho(t)}_{L^2_t L^2_x} \lesssim_{\sigma,\sigma',m} \norm{ h_{in}}_{H^\sigma_m} \label{ineq:LDdmpThm}
		\end{align} 
		\item if $k_3 = 0$ and we additionally have $\sigma > 5/2$, then for all $k_{\perp}$ and $n \in \Naturals$, $\exists !$ $b_{n,k} = b_{n,k}(\omega_c) \in ( n, n+1)$ and coefficients $r_{\pm n,k}$ depending on $h_{in}$ such that (with the convention that $r_{-0,k}$ is distinct from $r_{0,k}$),   
		\begin{align}
		\hat{\rho}(t,k_{\perp},0) = \sum_{n=0}^\infty r_{n,k} e^{i b_{n,k} \omega_c t } + r_{-n,k} e^{-i b_{n,k} \omega_c t }  \label{eq:BernExp}
		\end{align} 
		and further there holds: 
		\begin{align}
		\abs{r_{\pm n,k}} & \lesssim  \frac{1}{\brak{k}^{-\alpha} \brak{n}^{\gamma} }\norm{h_{in}}_{H^\sigma_m}, \label{ineq:ResEst}
		\end{align}
    for all $\alpha$, $\beta$, $\gamma$ such that $\alpha + \beta - \frac{1}{2} \leq \sigma$, $\alpha + 1 < m$, and $\gamma = \min(1,\beta-1)$. 
	\end{itemize} 
\end{theorem} 
\begin{remark}
	It is easy to extend the above theorem to analytic regularity and obtain exponential Landau damping in \eqref{ineq:LDdmpThm} (provided one takes stronger conditions on $\tilde{f}^3$).  
\end{remark}
\begin{remark}
	At least for $\sigma$ sufficiently large (e.g. $\sigma > 7/2$) it is straightforward to use \eqref{ineq:LDdmpThm} prove ``scattering'' to the passive transport equation. That is, there exists a solution $h_\infty$ to \eqref{pte} such that $\lim_{t \rightarrow \infty}\norm{h(t) - h_\infty(t)}_{L^2} = 0$ (the proper analogue of scattering in higher Sobolev norms is also straightforward, but is more complicated to state).  
\end{remark}

\begin{remark} 
At least for $\tilde{f^3} = 0$ and $T_{||} = 1$, from energy methods (see \S\ref{sec:Energy}), one can show that $\norm{\rho(t)}_{L^2} \lesssim \norm{h_{in}}_{L^2_\mu}$ (see \eqref{def:EnerL2} below for definition), however, it is not clear how to justify the Bernstein mode expansion \eqref{eq:BernExp} without some regularity. At least, more subtle harmonic analysis would be required. 
\end{remark} 
\begin{remark} 
There is not an explicit formula for $b_{n,k}$ in terms of $n,k$, and $\omega_c$, except in certain asymptotic limits (see \S\ref{sec:Bern} and \cite{Bernstein58,Stix,Swanson} 
\end{remark}
In the collisional case, the only equilibrium is the Maxwellian (we have taken the temperature one for simplicity without loss of generality) 
\begin{align}
f^0(v) = \mu(v) = \frac{1}{(2\pi)^{3/2}}e^{-\frac{\abs{v}^2}{2}}. 
\end{align}
Define the natural Gaussian weighted space that is the quadratic variation of the Boltzmann entropy and hence the natural energy for the collision operator
\begin{align}
\norm{f}_{L^2_\mu} = \left( \int_{\T^3 \times \Real^3} \frac{1}{\mu(v)} \abs{f(x,v)}^2 dx dv \right)^{1/2}. \label{def:EnerL2}
\end{align}
In the collisional case, we prove Theorem \ref{thm:Coll}. 
The proof of \eqref{ineq:coldec} follows from an energy argument that is a relatively simple variant of Yan Guo's energy methods found in \cite{Guo02,Guo03,Guo06,Guo12}. 
In the linear context, these methods essentially reduce to a type of hypocoercivity argument; see \cite{Villani2009,GallagherGallayNier2009} and the references therein for more standard hypocoercivity methods. 
The proof of \eqref{ineq:LDcol} follows from first reducing the problem again to a Volterra equation (note that it is not clear this is even possible in the case $\nu > 0$, especially when $b \neq 0$) and making a detailed study of this Volterra equation in the $\nu \rightarrow 0$ limit. 
\begin{theorem}[Collisional dynamics] \label{thm:Coll} 
	Then for all $\sigma \geq 0$, $\exists \nu_0 = \nu_0(\sigma)$ sufficiently small such that $\forall \nu \in (0,\nu_0(\sigma))$ (as usual, the implicit constants below are independent of $\nu$): 
	\begin{itemize} 
		\item there exists a universal $\delta > 0$ such that for all $\sigma \geq 1$, such that 
		\begin{align}
		\norm{\brak{\grad_x}^\sigma h(t)}_{L^2_\mu} \lesssim_\sigma  e^{-\delta \nu t}\norm{\brak{\grad_x}^\sigma h_{in}}_{L^2_\mu}. \label{ineq:coldec} 
		\end{align}
		\item there exists a universal $\delta' > 0$  such that for all $m \geq 2$ integers,  
		\begin{align}
		\norm{e^{ \delta' \nu^{1/3} t} \abs{\partial_{z}}^{1/2}\brak{\grad,\partial_z t}^\sigma \rho(t)}_{L^2_t L^2_x} \lesssim_{m,\sigma} \norm{h_{in}}_{H^\sigma_m}. \label{ineq:LDcol} 
		\end{align} 
	\end{itemize} 
\end{theorem} 
\begin{remark} \label{rmk:KolEqn}
To see the origin of the $\nu^{-1/3}$ time-scale, consider the Kolmogorov equation 
\begin{align}
\partial_t f + v \partial_x f= \nu \partial_{vv }f. 
\end{align}
Write $g(t,x,v) = f(t,x+tv,v)$ and this becomes
\begin{align}
\partial_t g = \nu (\partial_v - t \partial_x)^2 g. 
\end{align}
Applying the Fourier transform in both $x$ and $v$ gives 
\begin{align}
\widehat{g}(t,k,\eta) = \widehat{g}(0,k,\eta) \exp\left[-\nu \int_0^t \abs{\eta - k\tau}^2 d\tau \right]. 
\end{align}
From here we derive that for some universal $\delta > 0$, there holds 
\begin{align}
\abs{\widehat{g}(t,k,\eta)} \leq  \abs{\widehat{g}(0,k,\eta)} \exp \left[ -\nu \max(\eta^2t, \delta k^2 t^3) \right]. 
\end{align}
Notice that the rate is strongly dependent on the order of the smoothing operator in $v$. 
Indeed, the phase mixing is enhancing the collisions by producing large gradients in velocity, which are in turn dissipated faster than low frequencies. 
Hence, for plasma physics in the weak collisional regime, \emph{only second order} collision operators are suitable for making accurate predictions of time-scales. 

This effect was first discovered in the context of fluid mechanics by Kelvin \cite{Kelvin87}, where it is usually referred to as `relaxation enhancement' or `enhanced dissipation' (see e.g. \cite{CKRZ08,BCZ15} and the references therein). 
In fluid mechanics, it has been well-studied in both linear (see e.g. \cite{RhinesYoung83,LatiniBernoff01,BernoffLingevitch94,CKRZ08,BeckWayne11,VukadinovicEtAl2015,BCZ15,LiWeiZhang17} and the references therein) and some nonlinear problems (see e.g. \cite{BVW16,BMV14,BGM15I,BGM15II,BGM15III,WeiZhang18} and the references therein). 
\end{remark} 

\begin{remark} 
Enhanced dissipation and hypoellipticity are closely related in the context of kinetic theory. 
However, they are not equivalent. Indeed, the passive transport equation
\begin{align}
\partial_t h + v\cdot \grad_x h + (v \times B_0) \cdot \grad_v h = \nu \left( \Delta_v + \grad_v \cdot (vh) \right)
\end{align} 
is an example which is hypoelliptic (solutions will be instantly smooth in both $x$ and $v$), however, spatial modes $k$ for which $k_z= 0$ still decay only at the $O(\nu^{-1})$ time-scale. 
The reason is that the cyclotron motion of the collisionless dynamics, 
\begin{align}
\partial_t h + v\cdot \grad_x h + (v \times B_0) \cdot \grad_v h = 0
\end{align}
is periodic with period $2\pi \omega_c^{-1}$, however, higher modes in $k$ produce transient gradients in $v$ which, while bounded for each individual $k$, are unbounded for $k \rightarrow \infty$. Hence, for any fixed $k$, there is no change in the time-scale with respect to $\nu$, however, one nevertheless gets hypoellipticity due to the dependence of this time-scale on $k$. 
See \S\ref{sec:pte} and \S\ref{sec:VoltColl} for more details. 
\end{remark} 

\begin{remark}
The collision operator we consider in \eqref{def:VPElin} and Theorem \ref{thm:Coll} is not the linearization of the natural nonlinear Fokker-Planck operator that satisfies conservation of energy and momentum (see e.g. \cite{B17} and references therein). 
Nevertheless, the work of \cite{B17} treats the remaining non-local terms perturbatively using a combination of hypoellipticity and uniform-in-$\nu$ Landau damping. Due to the presence of non-Landau damping modes here, it is less clear that the linearized collision operator we consider will be sufficient for nonlinear studies, however, it is most likely the case for $k_z \neq 0$. 
\end{remark}

\begin{remark}
	Naturally, one is interested in extending the above results to the Landau collision operator. From \cite{Guo02,Guo12} it is clear that the energy method will apply to this general situation and one can likely deduce \eqref{ineq:coldec}. 
	However, the reduction to a Volterra equation used for the uniform-in-$\nu$ Landau damping and mixing-enhanced collisional relaxation does not seem to apply in this case. 
\end{remark}

\subsection{Notation} \label{sec:Note}
Throughout this paper, we use the Japanese bracket notation for a vector $v\in\mathbb R^3$ or $\mathbb T^3$
\begin{align}
\label{}
\brak{v} = \sqrt{1+|v|^2}.
\end{align}
We also write 
\begin{equation}
\label{}
v = (v_1, v_2, v_3) = (v_x, v_y, v_z) = (v_\perp, v_3)
\end{equation}
where $v_\perp$ is the horizontal component of $v$. We use both $v_z$ and $v_3$ interchangeably depending on which is more convenient.
The $H^m_n$ norm of a function $f$ is defined as follows, with the usual convention $L^2_m := H^0_m$:
\begin{align}
\label{}
\Vert f\Vert_{H^{\sigma}_m} ^2= \int_{\mathbb{T}^3\times\mathbb{R}^3} \brak{v}^{2m}|\brak{\nabla}^\sigma f|^2 \, dx \ddv,
\end{align}
where  $\grad_{x,v}$ is the differential operator in both $x$ and $v$. If $g$ is independent of $v$, $g=(t,x)$, then $\grad g$ reduces to only the derivatives in $x$.  
For $r\in\mathbb R$, we use $r^\pm$ to denote a real number close to $r$, i.e.,
\begin{equation}
\label{}
r^\pm=r\pm\delta
\end{equation}
where $\delta>0$ is any small number. We denote the Fourier transform of a function $f$ in both $x$ and $v$ variable by
\begin{equation}
\label{}
\hat{f}(k, \eta) = \int_{\mathbb T^3\times\mathbb R^3} f(x, v)\text{e}^{-ix\cdot k-iv\cdot\eta} \, dx\ddv
\end{equation}
for which the inverse Fourier transform is given by
\begin{equation}
\label{}
f(x, v) = \frac{1}{(2\pi)^6} \sum_{k\in\mathbb Z^3}\int_{\mathbb R^3} \hat f(k, \eta)\text{e}^{ix\cdot k+iv\cdot\eta} \, d\eta.
\end{equation}
A Fourier multiplier operator $m(\nabla)$ is defined via
  \begin{align}
  \label{}
  (m(\nabla)f)^{\hat{}}(k,\eta) = m(k,\eta) \hat{f}(k,\eta); 
  \end{align}
  and analogously for functions of $x$ only. 
We write the Laplace transform of a function $f(t)$ as
\begin{equation}
\label{}
\cL[f](z) = \int_0^\infty \text{e}^{-tz}f(t)\, dt,
\end{equation}
which is defined and holomorphic for $\Re z \geq \mu$ for $\mu > 0$ such that $e^{-\mu t} f(t) \in L^1_t$.  
If $\gamma \geq \mu$ and $\cL[f](\gamma \pm i \omega)$ is $L^1_\omega$, then the inverse transform is defined via the integration along the Bromwich contour
\begin{equation}
f(t) = \frac{1}{2\pi i} \int^{\gamma+i\infty}_{\gamma-i\infty} \text{e}^{tz}\cL[f](z)\, dz. \label{eq:invTrans}
\end{equation}
As in the case of the Fourier transform, this formula extends also to the case $\cL[f](\gamma \pm i \omega) \in L^{2}_\omega$ (and the analogue of Plancherel's identity still holds). 
We also recall that for $\Re z \geq \mu$ for $\mu$ sufficiently large, the Laplace transform of the Volterra equation 
\begin{align}
\rho(t) = f(t) + \int_0^t K(t-\tau) \rho(s) ds 
\end{align}  
is given by 
\begin{align}
\cL[\rho](z) = \cL[f](z) + \cL[K](z) \cL[\rho](z). 
\end{align}
Analytic continuation and contour deformation play important roles for making the Laplace transform useful for analyzing Volterra equations.

\section{Collisionless case}
The key structure in these kinetic problems is that we can reduce the full linearized dynamics to a Volterra equation. 
\subsection{Trajectories and Volterra equation reduction}
\subsubsection{Passive transport behavior}\label{sec:pte}
The passive transport equation reads
\begin{align}
\label{pte}
\partial_t h + v\cdot \grad_x h +  \frac{q}{m}v\times B_0 \cdot \grad_v h  = 0. 
\end{align}
We define the density associated with this equation as 
\begin{align}
\rho_0= \int_{\Real^3} h(t, x, v)\ddv.
\end{align}
The particle trajectories of \eqref{pte} solve the ODE (we have assumed for simplicity $q > 0$, the only difference in what follows is the direction of rotation) 
\begin{align}
\dot{X} & = V\\ 
\dot{V} & = \omega_c 
\begin{pmatrix} 
V_y \\ 
-V_x \\ 
0
\end{pmatrix}
. 
\end{align}
It is classical that the trajectories of charged particles in a uniform magnetic field trace out helices \cite{GoldstonRutherford95}.  
Denote the backwards characteristics ending at a point $(x,v)$ at time $t$ by $X(\tau;t,x,v), V(\tau;t,x,v)$, then 
\begin{align}
V_x(\tau) & = v_x \cos \omega_c(t-\tau) -  v_y \sin \omega_c(t-\tau) \\
V_y(\tau) & = v_x \sin \omega_c(t-\tau) + v_y \cos \omega_c(t-\tau) \\
V_z(\tau) & = v_z.
\end{align}
The position characteristics are given by
\begin{align}
X_x(\tau) & = x_x - \frac{v_x}{\omega_c} \sin \omega_c(t-\tau) + \frac{v_y}{\omega_c}(1- \cos \omega_c(t- \tau)) \\
X_y(\tau) & = x_y - v_x \frac{1}{\omega_c}(1- \cos \omega_c(t-\tau)) - \frac{v_y}{\omega_c} \sin \omega_c(t-\tau) \\
X_z(\tau) & = x_z - (t-\tau)v_z. 
\end{align}
It follows that the solution to \eqref{pte} is given by 
\begin{align}
h(t,x,v) & = h_{in} (X(0;t,x,v),V(0;t,x,v)). 
\end{align}
Taking the Fourier transform (in both variables) gives, 
\begin{align}
\widehat{h}(t,k,\eta) & = \int_{\TT^3 \times \Real^3} \hspace{-.5cm} e^{-i\eta \cdot v - ik\cdot x} h_{in}\Bigg(
\begin{pmatrix}
x_x - \frac{v_x}{\omega_c} \sin \omega_c t + \frac{v_y}{\omega_c}(1- \cos \omega_ct) \\
x_y - v_x \frac{1}{\omega_c}(1- \cos \omega_c t) - \frac{v_y}{\omega_c} \sin \omega_c t \\
x_z - tv_z
\end{pmatrix}
, 
\begin{pmatrix}
v_x \cos \omega_ct -  v_y \sin \omega_c t \\
v_x  \sin \omega_c t  + v_y \cos \omega_c t \\
v_z
\end{pmatrix} 
\Bigg)
dx dv \\
& = \int_{\TT^3 \times \Real^3} \hspace{-0.5cm} e^{-i \eta \cdot v - ik \cdot (x + \tilde{O}(t)v) } h_{in}\left(x, 
\begin{pmatrix}
v_x \cos \omega_ct - v_y \sin \omega_c t \\
v_x  \sin \omega_c t  + v_y \cos \omega_c t \\
v_z
\end{pmatrix} 
\right)dx dv, 
\end{align}
where
\begin{equation}
\label{}
\tilde{O}(t)
=
\begin{pmatrix}
\frac{1}{\omega_c}\sin(\omega_ct) & -\frac{1}{\omega_c}(1-\cos(\omega_ct) ) & 0\\
\frac{1}{\omega_c}(1-\cos(\omega_ct) ) & \frac{1}{\omega_c}\sin(\omega_ct) & 0\\
0 & 0 & t
\end{pmatrix}.
\end{equation}
Denote also the orthogonal matrix 
\begin{equation}
\label{}
O(t)
=
\begin{pmatrix}
\cos(\omega_ct) & -\sin(\omega_ct) & 0\\
\sin(\omega_ct) & \cos(\omega_ct) & 0\\
0 & 0 & 1
\end{pmatrix}. 
\end{equation}
Hence,
\begin{align}
\widehat{h}(t,k,\eta) 
& = \widehat{h_{in}} \left(k, O(t)\eta + O(t)\tilde{O}^{T}(t)k \right).  
\end{align}
It follows that the associated density  is given by: 
\begin{align}
\label{rho0}
\hat{\rho}_0(t,k) & = \widehat{h_{in}} \left(k, O(t)\tilde{O}^{T}(t)k \right).  
\end{align}
Next, we characterize the dynamics of this ``passive transport'' density. 
\begin{lemma} \label{lem:passive}
	Let $h$ solve the passive transport equation \eqref{pte} with initial data $h_{in}(x, v)$. Then the following holds: 

\begin{enumerate} 
  \item the Landau damping of $k_z \neq 0$ modes: for all $\sigma \geq 0$, $m > 2$, 
	\begin{align}
	\norm{\abs{\partial_z}^{1/2} \brak{\grad,\partial_z t}^\sigma \rho_0}_{L^2_t L^2_x} \lesssim_{\sigma,m} \norm{h_{in}}_{H^{\sigma}_m};  \label{ineq:LDrho0}
	\end{align}	
\item the decomposition into countably infinite oscillations at the cyclotron harmonics for $k_z = 0$: 
	for all $k$ with $k_z = 0$, there exist coefficients $g_{n,k}$ depending on $h_{in}$ such that  
	\begin{align}
	\label{rep:rho0}
	\hat{\rho_0}(t,k_{\perp},0) = \sum_{n=-\infty}^\infty g_{n,k} e^{in \omega_c t},
	\end{align}
	which satisfy the following for all $\alpha, \beta$ such that $\aa+\bb-1/2\le \sigma$ and $\aa +1 <  m$, 
	\begin{align}
	\label{ineq:gnkest}
	\sum_{n \in \Integer} \sum_{k_\perp \in \Integer^2} |n|^{2\aa}|k|^{2\bb}\abs{g_{n,k}}^{2} 
	\lesssim \Vert h_{in}\Vert_{H^{\aa+\bb-1/2}_m}^2 
	\end{align}
    and
	\begin{align}
	\label{Pla}
	\norm{\abs{\grad}^j \rho_0(t)}_{L^2([0, 2\pi/\omega_c]\times\mathbb T^3)}^2 \lesssim \sum_{n \in \Integer} \sum_{k_\perp \in \Integer^2} \abs{k^j g_{n,k}}^2.
	\end{align}
\end{enumerate} 
\end{lemma}
\begin{remark}
	 Note that \eqref{rep:rho0} implies $\rho_0$ is periodic with period $2\pi \omega_c^{-1}$. 
\end{remark}
\begin{proof}
	The velocity averaging estimate \eqref{ineq:LDrho0} follows via the Sobolev trace lemma on the Fourier side (in particular, that the $L^2$ norm of the trace along a line is bounded by the $H^{(d-1)/2}$ norm) 
	\begin{align}
	\int_0^\infty \sum_{k \in \Integer^3_\ast} \abs{k_z} \abs{\widehat{h_{in}} \left(k, O(t)\tilde{O}^{T}(t)k \right)}^2 dt & \lesssim \sum_{k \in \Integer^3_\ast} \sup_{\eta \in \Real^2}   \int_0^\infty \abs{k_z} \abs{\widehat{h_{in}}(k,\eta_2,k_z t)}^2 dt \\
	& \lesssim \sum_{k \in \Integer^3} \sup_{\eta_2 \in \Real^2} \int_{-\infty}^\infty \abs{\widehat{h_{in}}(k,\eta_2, \eta_z)}^2 d\eta_z \\
	& \lesssim \sum_{k \in \Integer^3} \int_{\Real^3} \sum_{\abs{\alpha} \leq m} \abs{D_\eta^\alpha \widehat{h_{in}}(k,\eta)}^2 d\eta. 	
	\end{align}
	The inclusion of the Fourier multiplier $\brak{\grad,\partial_z t}$ is immediate, which concludes the proof of \eqref{ineq:LDrho0}. 
	
	Turn next to the expansion of $\rho_0$ in the case that $k_z = 0$. 
	Recall   
	\begin{align} 
	O\tilde{O}^T(t)k & = 
	\begin{pmatrix}
	\frac{1}{\omega_c}\sin(\omega_ct) & -\frac{1}{\omega_c}(1-\cos(\omega_ct) ) & 0\\
	\frac{1}{\omega_c}(1-\cos(\omega_ct) ) & \frac{1}{\omega_c}\sin(\omega_ct) & 0\\
	0 & 0 & t
	\end{pmatrix} k \\ 
	& = \begin{pmatrix} 
	\frac{k_x}{\omega_c}\sin(\omega_ct) - \frac{k_y}{\omega_c}(1-\cos(\omega_ct) ) \\ 
	\frac{k_x}{\omega_c}(1-\cos(\omega_ct) ) + \frac{k_y}{\omega_c}\sin(\omega_ct) \\ 
	0
	\end{pmatrix} \\ 
	& = \begin{pmatrix} 
	-\frac{k_y}{\omega_c} \\ 
	\frac{k_x}{\omega_c} \\ 
	0
	\end{pmatrix} 
	+ 
	\begin{pmatrix} 
	\frac{k_x}{\omega_c}\sin(\omega_ct) + \frac{k_y}{\omega_c}\cos(\omega_ct) \\ 
	-\frac{k_x}{\omega_c} \cos(\omega_ct)  + \frac{k_y}{\omega_c}\sin(\omega_ct) \\ 
	0
	\end{pmatrix} 
.
	\end{align} 
	Hence, we see that $O\tilde{O}^T(t)k$ traces out a circle of radius $\abs{k_{\perp}}\omega_c^{-1}$  centered at $\omega_c^{-1}(-k_y,k_x)$ over a period of $2\pi/\omega_c$. 
	In particular, we can find $\phi = \phi(k)$ such that 
	\begin{align}
	O\tilde{O}^T(t)k = \begin{pmatrix} 
	-\frac{k_y}{\omega_c} \\ 
	\frac{k_x}{\omega_c} \\ 
	0
	\end{pmatrix} 
	+ 
	\frac{\abs{k_{\perp}}}{\omega_c} 
	\begin{pmatrix} 
	\cos(\omega_c t + \phi) \\ 
	\sin(\omega_c t + \phi) \\ 
	0
	\end{pmatrix} 
	.
	\end{align}
	Write in polar coordinates: 
	\begin{align}
	g_k(\theta) = \widehat{h_{in}}\left(k, \frac{\abs{k_{\perp}}}{\omega_c}\begin{pmatrix}  \cos(\theta + \phi)  \\ \sin(\theta + \phi) \\ 0 \end{pmatrix} + \begin{pmatrix} 
	-\frac{k_y}{\omega_c} \\ 
	\frac{k_x}{\omega_c} \\ 
	0
	\end{pmatrix}  \right)
	\end{align}
	and hence by definition and \eqref{rho0}, 
	\begin{align}
	\hat{\rho}_0(t,k) & = g_k(\omega_c t). 
	\end{align}
	Then, \eqref{rep:rho0} is obtained by expanding $g_k(\theta)$ in an angular Fourier series: 
	\begin{align}
	g_{k}(\theta) = \sum_{n \in \Integers} g_{k,n} e^{in\theta}. 
	\end{align}
	By Plancherel, we have for $\aa\in \mathbb R$
	\begin{align}
	\norm{\abs{\partial_\theta}^\aa g_k}_{L^2_\theta}^2 = \sum_{n \in \Integers} n^{2\aa}\abs{g_{n,k}}^{2}. 
	\end{align}
	Hence, \eqref{Pla} for $j=0$ follows. 
	We denote \[O\tilde{O}^Tk([a, b])=\{x\in\mathbb R^3: x=O\tilde{O}^T(t)k\ \text{for\ some}\ t\in[a, b]\}.\]
	Due to the Sobolev-trace lemma, we have for all $\alpha \geq 0$ and $m > \alpha + 1$
	\begin{align}
	\norm{\abs{\partial_\theta}^\aa g_k}_{L^2_\theta[0, 2\pi/\omega_c]} & 
	\lesssim |k|^{\alpha-1/2} \norm{\brak{D_\eta}^{\aa} \widehat{h_{in}}(k,\cdot)}_{L^2_\eta( O\tilde{O}^Tk([0, 2\pi/\omega_c]))}
	\nonumber\\&
	\lesssim \abs{k}^{\aa-1/2} \norm{\brak{D_\eta}^{m} \widehat{h_{in}}(k,\cdot)}_{L^2_\eta(\mathbb R^3)}.
	\end{align}
	Note the powers of $k$ come from the Jacobian and the chain rule. 
	Hence by Plancherel again, we get
	\begin{align}
	\sum_{n \in \Integer} \sum_{k_{\perp} \in \Integer^2} n^{2\aa}|k|^{2\bb}\abs{g_{n,k}}^{2} 
	& \lesssim 	\sum_{n \in \Integer} \sum_{k_{\perp} \in \Integer^2} \abs{k}^{2(\aa+\bb)-1} \norm{\brak{D_\eta}^{\alpha+1^+} \widehat{h_{in}}(k,\cdot)}_{L^2_\eta}^2
	\nonumber\\&
	\lesssim \Vert h_{in}\Vert_{H^{\sigma}_{m}}^2
	\end{align}
	concluding the proof.
\end{proof}

\subsubsection{Non-local response} 
Next, we will study the full linearized Vlasov equations, 
\begin{align}
\partial_t h + v\cdot \grad_x h +  \frac{q}{m} v\times B_0 \grad_v h  = -\frac{q}{m} E(t,x) \cdot \grad_v f^0. 
\end{align}
Using the trajectories introduced above for the passive transport equation \eqref{pte}, the method of characteristics gives
\begin{equation}
\label{}
\frac{d}{d\tau} h(t, X(\tau; t, x, v), V(\tau; t, x, v)) = -\frac{q}{m} E(t, X(\tau; t, x, v)) \cdot \grad_v f^0(V(\tau; t, x, v)),
\end{equation}
and hence 
\begin{equation}
\label{}
h(t, x, v) = h_{in}(X(0; t, x, v), V(0; t, x, v)) - \int_0^t \frac{q}{m} E(t, X(\tau; t, x, v)) \cdot \grad_v f^0(V(\tau; t, x, v)) \, d\tau.
\end{equation}
Taking the Fourier transform both in $x$ and $v$ as in Section~\ref{sec:pte}  gives 
\begin{align}
\widehat{h}(t,k,\eta) 
& = \widehat{h_{in}} \left(k, O(t)\eta + O(t)\tilde{O}^{T}(t)k \right) 
\nonumber\\&\quad
+
\frac{qi}{m}\int_0^t  \hat\rho(\tau,k)\widehat{W}(k)k\cdot(\widehat{\nabla_vf^0})(O(t-\tau)\eta + O\tilde{O}^T(t-\tau)k)\,d\tau.
\end{align}
Therefore, the Volterra equation for $\rho$ reads (recall $\widehat{\rho_0}$ is the density associated with the passive transport \eqref{pte}):
\begin{align}
\hat\rho(t,k) 
& =: \hat \rho_{0}\left(t,k\right)  + \int_0^t \hat{\rho}(\tau,k) K(t-\tau,k) d\tau \label{equ:vol}
\end{align}
where
\begin{align}
K(t,k) := -\frac{iq}{m}\widehat{W}(k)k\cdot(\widehat{\nabla_vf^0})(O\tilde{O}^T(t)k).
\end{align}
To further expand $K$, first note  
\begin{align}
\label{}
\left|O\tilde{O}^T(t) k_{\perp}\right|^2=&
\left(\frac{k_2}{\omega_c}\cos(\omega_c t)+\frac{k_1}{\omega_c}\sin(\omega_c t)-\frac{k_2}{\omega_c}\right)^2
+
\left(\frac{k_2}{\omega_c}\sin(\omega_c t)-\frac{k_1}{\omega_c}\cos(\omega_c t)+\frac{k_1}{\omega_c}\right)^2
\nn\\
&=2\frac{k_1^2+k_2^2}{\omega_c^2}-2\frac{k_1^2+k_2^2}{\omega_c^2}\cos(\omega_c t)\nn\\
&=2\frac{|k_{\perp}|^2}{\omega_c^2}-2\frac{|k_{\perp}|^2}{\omega_c^2}\cos(\omega_ct)
\end{align}
where recall we denote $|k_{\perp}|^2=k_1^2+k_2^2$ and 
\begin{align}
\label{}
k\cdot O\tilde{O}^T(t)k&=(k_1, k_2, k_3)
\begin{pmatrix}
\frac{1}{\omega_c}\sin(\omega_c t) & -\frac{1}{\omega_c}(1-\cos(\omega_c t) ) & 0\\
\frac{1}{\omega_c}(1-\cos(\omega_c t) ) & \frac{1}{\omega_c}\sin(\omega_c t) & 0\\
0 & 0 & t
\end{pmatrix}
\begin{pmatrix}
k_1\\
k_2\\
k_3
\end{pmatrix} \nn\\
&
=\frac{k_1^2}{\omega_c}\sin(\omega_c t)+\frac{k_2^2}{\omega_c}\sin(\omega_c t)+k_3^2 t\nn\\
&
=\frac{|k_{\perp}|^2}{\omega_c}\sin(\omega_c t)+k_3^2 t.
\end{align}
Hence, the $K(t,k)$ defined in \eqref{equ:vol} is written by 
\begin{align}
\label{source}
K(t,k) & = -\frac{q}{m} \widehat{W}(k)k\cdot O\tilde{O}^T(t) k\exp{\left(2\frac{|k_{\perp}|^2}{\omega_c^2}-2\frac{|k_{\perp}|^2}{\omega_c^2}\cos(\omega_ct)\right)} \widehat{f^0_3}(k_3t) \nonumber \\ 
& = -A_k \sin(\omega_c t)\exp{\left(-2\frac{|k_{\perp}|^2}{\omega_c^2}\cos(\omega_c t)\right)}\widehat{f^0_3}(k_3t) \nonumber  \\
&\quad - B_k  t\exp{\left(-2\frac{|k_{\perp}|^2}{\omega_c^2}\cos(\omega_c t)\right)} \widehat{f^0_3}(k_3 t), 
\end{align}
where 
\begin{equation}
\label{Ak}
A_k = 
\frac{q}{m}\widehat{W}(k)\frac{|k_{\perp}|^2}{\omega_c}\exp{\left(-2 \frac{|k_{\perp}|^2}{\omega_c^2}\right)}
\end{equation}
and
\begin{equation}
\label{}
B_k
=
\frac{q}{m}\widehat{W}(k)k_3^2\exp{\left(-2 \frac{|k_{\perp}|^2}{\omega_c^2}\right)}.
\end{equation}

Given that \eqref{equ:vol} together with \eqref{source} defines a Volterra equation for each $k$, it is very natural to apply Laplace transform methods (see \S\ref{sec:Note} for the conventions used). It is not hard to verify that the Laplace transform of $\rho$ and $\rho_{0}$ are holomorphic for $\textup{Re} z > C$ for some sufficiently large $C \geq 0$ and that $\cL[K](z,k)$ is holomorphic over $\textup{Re} z > 0$. Hence, over the half-plane $\textup{Re} z > C$, there holds
\begin{align}
\cL[\hat\rho](z,k) = \cL [\hat\rho_{0}](z,k) + \cL[\hat\rho](z,k) \cL [K](z,k).
\end{align} 
Below we also use the notation 
\begin{align}
L(z,k) := \cL [K](z,k).
\end{align}
The points $z \in \Complex$ where $1 = L(z,k)$ are called \emph{solutions of the dispersion relation}. Points $z_0 \in \Complex$ where $\lim_{z \rightarrow z_0}\abs{L(z,k)} = \infty$ are known as \emph{resonances}. 
Through the inverse Laplace transform, these two sets will determine most of the important properties of $\hat{\rho}(t,k)$.  

\subsection{Bernstein modes: the $k_z= 0$ case} \label{sec:Bern}
First, we consider modes with $k_z = 0$, i.e. completely transverse to the magnetic field. 
Set 
  \begin{equation}
  \label{def:a}
  a=2 |k_{\perp}|^2/\omega_c^2
  \end{equation}
and let $I_n(x)$ be the generalized Bessel functions defined by
\begin{equation}
\label{fun:Bes}
I_n(x)=\sum_{m=0}^\infty \frac{1}{m!\Gamma(m+n+1)}\left(\frac{x}{2}\right)^{2m+n}. 
\end{equation}
We begin with some preliminary lemmas, the first two are due to Bernstein \cite{Bernstein58} (and the third essentially is as well), but we include the proofs for readers' convenience. 
The first lemma characterizes $L(z,k)$. 
\begin{lemma}
	For $k_z = 0$,  $L(z,k)$ is holomorphic for all $z \neq \pm i n \omega_c$ and is given by the following (absolutely convergent for $z \neq \pm i n\omega_c$) 
	\begin{align}
	\label{Lap:K}
	\cL [K](z,k) = -A_k \sum_{n=1}^\infty \frac{2n}{a} I_n(a) \frac{n\omega_c}{z^2 + (n\omega_c)^2}. 
	\end{align}
\end{lemma}
\begin{proof}
In the case $k_z = 0$, the $K(t,k)$ simplifies to 
\begin{align}
K(t,k) = -A_k \sin(\omega_c t)\exp{\left(2 \frac{|k_{\perp}|^2}{\omega_c^2}\cos(\omega_c t)\right)}. 
\end{align}
	By \eqref{Besid0} we get
	\begin{align}
	\label{}
	& \sin(\omega_ct)\exp{\left(2 \frac{|k_{\perp}|^2}{\omega_c^2}\cos(\omega_ct)\right)}
	=
	\sin(\omega_ct)(I_0(a)+2\sum_{n=1}^\infty I_n(a)\cos(n\omega_ct)\nn\\
	&\quad
	=I_0(a)\sin(\omega_ct)+2\sum_{n=1}^\infty I_n(a)\sin(\omega_ct)\cos(n\omega_ct)\nn\\
	&\quad
	=
	I_0(a)\sin(\omega_ct)+\sum_{n=1}^\infty I_n(a)(\sin((n+1)\omega_ct)-\sin((n-1)\omega_ct))\nn\\
	&\quad
	=
	\sum_{n=1}^\infty(I_{n-1}(a)-I_{n+1}(a))\sin(n\omega_ct)=\sum_{n=1}^\infty \frac{2n}{a}I_n(a)\sin(n\omega_ct), 
	\end{align}
	where we also used the identity \eqref{Besid1} in the last equality. 
    Note that by \eqref{1st:mom}, the sum is absolutely convergent for all $t$. 
    Then, for $\Re z > 0$, the Laplace transform is explicitly computed as 
	\begin{align}
	\label{Lap}
	\mathcal{L}\left[\sin(\omega_ct)\exp{\left(2 \frac{|k_{\perp}|^2}{\omega_c^2}\cos(\omega_ct)\right)}\right]&=
	\sum_{n=1}^\infty \frac{n}{ai}I_n(a)\left(\frac{1}{z-in\omega_c}-\frac{1}{z+in\omega_c}\right )
	\nn\\&
	=\sum_{n=1}^\infty \frac{2n}{a}I_n(a)\frac{n\omega_c}{z^2+(n\omega_c)^2}.
	\end{align}
    By \eqref{1st:mom}, the sum is absolutely convergent and defines a holomorphic function for all $z \neq \pm i n\omega_c$. 
\end{proof}

From \eqref{Lap:K}, we see that there is a resonance at each cyclotron harmonic $z = \pm in\omega_c$ (recall we refer to a `resonance' as $z_0 \in \Complex$ such that $\lim_{z \rightarrow z_0}\abs{L(z,k)} = \infty$. 
The location of solutions of the dispersion relation -- that is, points $z_0 \in \Complex$ such that $1 = L(z_0,k)$ -- is less clear. 
The next lemma shows that such points can only exist on the imaginary axis. 

\begin{lemma}
	Every $z$ such that $L(z,k) = 1$ satisfies $\textup{Re} z = 0$.  
\end{lemma} 
\begin{proof} 
	We first claim that a solution to the dispersion relation can only appear on the imaginary axis or on the real axis. In fact, if $z\in\mathbb C$ is not in the two axes, then $\textup{Im} z^2\neq0$. According to \eqref{Lap:K}, $L$ is given by
	\begin{equation}
	\label{}
	L(z, k)=-A_k\sum_1^\infty \frac{2n}{a}I_n(a)\frac{n\omega_c}{z^2+(n\omega_c)^2}.
	\end{equation}
	Note that if  $\textup{Im} z^2\neq0$, then
	\begin{equation}
	\label{}
	\textup{Im} \frac{n\omega_c}{z^2+(n\omega_c)^2} \neq 0
	\end{equation} 
	and it has the opposite sign as $\textup{Im} z^2$. Therefore, we deduce
	\begin{equation}
	\label{}
	\textup{Im} L \neq 0,
	\end{equation}
	hence $L\neq1$. 
	Since $L<0$ on the real axis (recall \eqref{Ak}), no solutions to the dispersion relation can occur on the real axis. 
Hence, the lemma follows.
\end{proof} 

The next lemma characterizes solutions to the dispersion relation as simple poles of $(1-L)^{-1}$ on the imaginary axis between each cyclotron harmonic.  
\begin{lemma}[Bernstein modes]  \label{lem:Bern}
	For each $k_{\perp}\in\mathbb Z^2_\ast$ and $n \in \Naturals$, there exists a unique $b_n = b_n(\omega_c,k) \in (n,n+1)$ such that $L(\pm i \omega_c b_n,k) = 1$. More specifically, $(1-L)^{-1}$ has a simple pole at each $\pm i\omega_c b_{n}$. 
\end{lemma} 
\begin{proof} 
	On the imaginary axis, we have by \eqref{Lap}
	\begin{align}
	\label{}
	L(iy, k) = -A_k\sum_{n=1}^\infty \frac{2n}{a}I_n(a) \frac{n\omega_c}{(n\omega_c)^2-y^2}.
	\end{align}
	Since $L(iy, k)$ is even in $y$, we just need to consider $y\ge0$.
	Taking the derivative with respect to $y$ we arrive at
	\begin{align}
	\label{}
	\partial_y L = -A_k\sum_{n=1}^\infty \frac{2n}{a}I_n(a) \frac{2n\omega_cy}{((n\omega_c)^2-y^2)^2}<0
	\end{align}
	for $y\neq n\omega_c$ and $y>0$.
	Noting that \[\lim_{y \searrow n\omega_c} L(iy,k) = +\infty\] 
	and \[\lim_{y \nearrow (n+1)\omega_c} L(iy,k) = -\infty,\]
	the proof is completed.
\end{proof}  
\begin{remark} \label{rmk:Resonance}
The decomposition given in \ref{eq:BernExp} is \emph{not} an immediate consequence of Lemma \ref{lem:Bern}, not even formally. 
Formally, the Laplace transform of the density is given by the following formula wherever $\cL[K] \neq 1$: 
\begin{align}
\cL[\hat\rho](z,k) = \frac{\cL[\hat\rho_0](z,k)}{1-\cL[K](z,k)}.  
\end{align}
Given the poles in $\cL[\hat\rho_0](z,k)$ (see \eqref{rep:rho0}), if $\cL[K]$ were bounded, than $\hat{\rho}(t,k)$ would (formally) contain oscillations at  
the cyclotron harmonics $\pm i n \omega_c$ \emph{and} the Bernstein modes $\pm i b_n \omega_c$. 
The decomposition in \eqref{eq:BernExp} is only possible due to the resonances at the cyclotron harmonics. 
\end{remark}

Next, we study the Laplace transform of $\hat{\rho}(z,k)$.
In particular, we show that the resonances cancel the poles present in $\cL[\hat{\rho}_0]$ from \eqref{rep:rho0}, ultimately making a decomposition like \eqref{eq:BernExp} possible. 
\begin{lemma} \label{lem:poles}
	The Laplace transform of the density $\mathcal{L}[\hat{\rho}](z,k)$ is holomorphic at all $z \in \Complex$ \emph{except} $z = \pm i \omega_cb_{n}$, $n \in \Naturals_\ast$, where $\cL[\hat{\rho}]$ has a simple pole.
\end{lemma} 
\begin{proof}
For $\Re z > C$, we have the formula (provided the $g_{n,k}$ are summable)
	\begin{align}
	\mathcal{L}[\hat{\rho}](z,k) & = \left(\sum_{n \in \Integers} g_{n,k} \frac{1}{in\omega_c-z} \right) \left(\frac{1}{1-iA_k\sum_1^\infty \frac{n}{a}I_n(a)(\frac{1}{z-in\omega_c}-\frac{1}{z+in\omega_c})} \right). \label{eq:Lrho}
	\end{align}
	We claim that $\mathcal{L}[\hat\rho]$ is holomorphic for all $ z \neq \pm i \omega_c b_{\ell}$. 
	Away from the Bernstein modes and the cyclotron harmonics, \eqref{eq:Lrho} is the product of two holomorphic functions, and hence by analytic continuation we can extend this formula for $\cL[\hat\rho](z,k)$ to all $z \neq i n \omega_c$ and $z \neq i b_n \omega_c$. 
Next, we show that $\mathcal{L} [\hat{\rho}](z,k)$ is actually holomorphic at the cyclotron harmonics. 
	Indeed,  	suppose that $z \rightarrow i \omega_c \ell$ for some $\ell \in \Integer$. We see that $\mathcal{L} [\hat{\rho}](z,k)$ is continuous (i.e. the singularity is removable): 
	\begin{align}
	\lim_{z \rightarrow i\omega_c \ell} \mathcal{L} [\hat{\rho}](z,k) = \frac{g_{\ell,k}}{iA_k \frac{\ell}{a}I_\ell(a)}. 
	\end{align}
	Something analogous holds for the derivatives as well, and hence $\mathcal{L} [\hat{\rho}]$ is holomorphic also at the cyclotron harmonics. 	

    Finally, it is evident from \eqref{eq:Lrho} that $\mathcal{L}[\hat{\rho}]$ has a simple pole whenever $z = \pm i\omega_cb_\ell$ for one of the Bernstein modes $b_\ell$.  
\end{proof}

In order to make rigorous the decomposition of $\widehat{\rho}(t,k)$ suggested by Lemma \ref{lem:poles}, we begin by analyzing the residues at the Bernstein modes and deducing 
decay estimates in terms of both spatial mode $k$ and in cyclotron harmonic number. 
Note, 
\begin{align}
\partial_z L(z,k) = A_k \sum_{n=1}^\infty \frac{4nz}{a} I_n(a) \frac{n\omega_c}{(z^2 + (n\omega_c)^2)^2}. 
\end{align}

\begin{lemma}[Residue estimates]
	\label{Residues}
	The residues of $\cL [\hat\rho](z,k)$ as $z \rightarrow \pm i\omega_cb_{\ell, k}$ are given by $e^{\pm i\omega_c b_\ell t } r_{\pm \ell,k}$ where (with the convention that $r_{0,k}$ and $r_{-0,k}$ are distinct)
	\begin{align}
	\label{res:for}
	r_{\pm \ell,k} &:= \frac{1}{-\partial_zL(z,k)}\left(\sum_{n \in \Integers} g_{n,k} \frac{1}{in\omega_c-z} \right)\Bigg|_{z= \pm i\omega_c b_{\ell}},
	\end{align}  
	which are well-defined (if $\sigma> 5/2$ and $m > 2$) and satisfy the following estimate for all $\alpha$ and $\beta$ with  $\aa+\bb-1/2\le \sigma$ and $\aa+1 < m$: 
	\begin{align}
	\label{est:res}
	\abs{r_{\pm \ell,k}} \lesssim \brak{k}^{-\aa}\brak{\ell}^{-\min(1,\beta-1)}\Vert h_{in}\Vert_{H^{\sigma}_m}.
	\end{align}
\end{lemma} 

Assuming \eqref{res:for} is true, we prove the estimate of $r_{\ell,k}$ in \eqref{est:res} in the following two lemmas where we deal with the cases $a$ relatively small and relatively large separately
(recall the definition of $a$ in \eqref{def:a}). 
 
\begin{lemma}[Low frequency]
	\label{small}
	Let $a_0>0$ be any fixed number, then the estimate \eqref{est:res} holds for $a\le a_0$ (with implicit constant depending on $a_0$).
\end{lemma} 

\begin{proof} 
	Let $a_0\in\mathbb R^+$. Without loss of generality,  we assume that there exists at least one $k_{\perp}$ such that $0<a=2 |k_{\perp}|^2/\omega_c^2<a_0$ since otherwise the result is automatically true. Note that for any $a>0$ and $n\ge a$, one has
	\begin{align}
	\label{}
	\frac{1}{m!(m+n)!}\left(\frac{a}{2}\right)^{2m+n}
	= \left(\frac{a}{2}\right)^n\frac{1}{n!}\frac{(a/2)^{2m}}{(n+1)\cdots(n+m)m!} \le \left(\frac{a}{2}\right)^n\frac{1}{n!}\frac{1}{2^m}\frac{(a/2)^m}{m!}.
	\end{align}
	Recalling definition \eqref{fun:Bes}, summing up over $m$, and appealing to Stirling's formula, we arrive at
	\begin{equation}
	\label{bd:In1}
	I_n(a) \le \left(\frac{a}{2}\right)^n\frac{2}{n!}\exp{\left(\frac{a}{2}\right)}
	\end{equation}
	and
	\begin{equation}
	\label{bd:In2}
	I_n(a) \lesssim\left(\frac{\text{e}a}{2n}\right)^n\exp{\left(\frac{a}{2}\right)} \ \ \ \ \text{for}\ n\ \text{large\ and\ each\ } a>0.
	\end{equation}
	Recall \eqref{Lap:K} and that on the imaginary axis $z=iy$ ($y \in \Real$), $L(iy, k)$ is an even function of $y$. Therefore, it is sufficient to consider the case $y>0$. Setting $y=\omega_c b$, we obtain
	\begin{equation}
	\label{for:L}
	L(i\omega_cb,k) = -A_k\sum_{n=1}^\infty \frac{2n}{a}\frac{1}{\omega_c}I_n(a) \frac{n}{n^2-b^2}.
	\end{equation}
	Let  $b_{\ell}\in(\ell-1/2, \ell+1/2)$ be a Bernstein mode (note that this is not necessarily the same enumeration that we used above). We want to show that there exists $n_0\in\mathbb{N}$ sufficiently large such that for $\ell \geq n_0$, the $\ell$-th term in \eqref{for:L} is $\gtrsim 1$. We decompose the sum of all other terms in three pieces: 
	\begin{align}
	A_k \sum_{\substack{n\neq \ell\\n\ge1}}\frac{2n}{a}\frac{1}{\omega_c}I_n(a) \frac{n}{n^2-b_{\ell}^2}
	&=	A_k\left(\sum_{1\le n\le a}+\sum_{a<n\le N(a)}+\sum_{\substack{N(a)<n\\ n\neq \ell}}\right)\frac{2n}{a}\frac{1}{\omega_c}I_n(a) \frac{n}{n^2-b_{\ell}^2}
	\nonumber\\&
	=S_1+S_2+S_3, \label{def:SiLow}
	\end{align}
	where $N(a)>a$ is a sufficiently large number to be determined below. From \eqref{e:a}, it follows that 
	\begin{align}
	\label{}
	|S_1| \le A_k\frac{2}{\omega_c}\frac{a}{b_{\ell}(b_{\ell}-a)}\sum_{n\le a}I_n(a) \le  A_k\frac{2}{\omega_c}\frac{a}{b_{\ell}(b_{\ell}-a)}\exp{\left(a\right)}.
	\end{align}
	In light of \eqref{Ak}, we further obtain
	\begin{equation}
	\label{bd:s1}
	|S_1| \lesssim \frac{a}{b_{\ell}(b_{\ell}-a)}.
	\end{equation}
	Using \eqref{bd:In1}, we bound the second term in \eqref{def:SiLow} as
	\begin{align}
	\label{bd:s2}
	|S_2| &\le
	A_k\frac{N(a)}{b_{\ell}^2-N(a)^2}\frac{1}{\omega_c}\sum_{a<n\le N(a)}\frac{2n}{a}\left(\frac{a}{2}\right)^n\frac{2}{n!}\exp{\left(\frac{a}{2}\right)}
	\lesssim 	\frac{aN(a)}{b_{\ell}^2-N(a)^2}.
	\end{align}
	For $S_3$, we employ \eqref{bd:In2} to obtain
	\begin{align}
	\label{bd:s3}
	|S_3| & \lesssim
	A_k\sum_{\substack{N(a)<n\\ n\neq \ell}} \frac{n}{n^2-b_{\ell}^2}\left(\frac{\text{e}a}{2n}\right)^{n-1}\exp{\left(\frac{a}{2}\right)}
	\lesssim
	A_k\sum_{\substack{N(a)<n\\ n\neq \ell}} \left(\frac{\text{e}a}{2n}\right)^{n-1}\exp{\left(\frac{a}{2}\right)}
	\nonumber\\&
	\lesssim
	\left(\frac{\text{e}a}{2N(a)+2}\right)^{N(a)}\exp{\left(-\frac{a}{2}\right)}. 
	\end{align}
	We first choose $N(a)\in\mathbb{N}$ large enough such that $|S_3|\le 1/4$. Then we choose $n_0\ge N(a)$ sufficiently large to ensure 
	$|S_1|, |S_2|\le 1/4$ for $\ell\ge n_0$. It follows that (using that $L(i \omega_c b_\ell,k) = 1$)
	\begin{equation}
	\label{lterm}
	A_k\frac{2\ell}{a}\frac{1}{\omega_c}I_{\ell}(a) \frac{\ell}{|\ell^2-b_{\ell}^2|} \ge \frac{1}{4},
	\end{equation}
	which implies
	\begin{equation}
	\label{pole:dis}
	0<\ell-b_{\ell}\le 5 A_k\frac{\ell}{a}\frac{1}{\omega_c}I_{\ell}(a)
	\lesssim \left(\frac{\text{e}a}{2\ell}\right)^{\ell-1}\exp{\left(-\frac{a}{2}\right)}.
	\end{equation}
	As we will see below, this will imply a large lower bound on the $\partial_yL(iy,k)$, which is the key to prove that the sequence of residues is summable. 
	By \eqref{res:for}, we obtain
	\begin{align}
	\label{res:Rbd}
	&|r_{\ell,k}|=\left|\text{Res}_{ib_{\ell}\omega_c}\frac{1}{1-L}\left(\sum_{n \in \Integers} g_{n,k} \frac{1}{in\omega_c-z} \right)\right| \lesssim 
	\left|\frac{1}{\partial_{y}L(ib_{\ell}\omega_c)}\left(\sum_{n \in \Integers} g_{n,k} \frac{1}{n-b_{\ell}} \right)\right|
	\nonumber\\&
	\qquad
	\lesssim \left|\frac{1}{\partial_{y}L(ib_{\ell}\omega_c)}\left(\sum_{|n-\ell|\ge n/2} g_{n,k} \frac{1}{n-b_{\ell}} \right)\right|
	+
	\left|\frac{1}{\partial_{y}L(ib_{\ell}\omega_c)}\left(\sum_{0<|n-\ell|< n/2} g_{n,k} \frac{1}{n-b_{\ell}} \right)\right|
	\nonumber\\&
	\qquad\quad
	+
	\left|\frac{1}{\partial_{y}L(ib_{\ell}\omega_c)}\left(g_{\ell,k} \frac{1}{\ell-b_{\ell}} \right)\right| 
	\nonumber \\ 
	& \qquad =:R_1+R_2+R_3.
	\end{align}
        We first deal with $R_1$. Note that by \eqref{ineq:gnkest} we have (using $\sigma>5/2$, $m>2$)
	\begin{equation}
	\label{}
	\left|\sum_{|n-\ell|\ge n/2} g_{n,k} \frac{1}{n-b_{\ell}}\right| \lesssim \frac{\Vert h_{in}\Vert_{H^{\sigma}_m}}{\ell}
	\end{equation}
	for $\ell>0$. A direct calculation gives a lower bound of the derivative of $L$:
	\begin{align}
	\label{}
	|\partial_yL(i\omega_c b_\ell )|
	\ge A_k\frac{2\ell}{a}\frac{1}{\omega_c^2}I_{\ell}(a) \frac{2\ell b_{\ell}}{(\ell^2-b_{\ell}^2)^2}
	\end{align}
	for any $\ell\in\mathbb N$. Using \eqref{lterm}, we arrive at
	  \begin{align}
	  \label{der:L}
	  |\partial_yL(i \omega_c b_\ell)| \gtrsim \frac{b_{\ell}}{|\ell^2-b_{\ell}^2|}
  	  \end{align}
	for $\ell\ge n_0$.
	Hence by \eqref{pole:dis} and \eqref{der:L} we deduce that
	\begin{equation}
	\label{}
	R_1 \lesssim \frac{1}{\ell}\left(\frac{1}{\ell}\right)^{\ell-1} \Vert h_{in}\Vert_{H^{\sigma}_m}
	=\left(\frac{1}{\ell}\right)^{\ell}\Vert h_{in}\Vert_{H^{\sigma}_m}
	\end{equation}
	for $\ell\ge n_0$.
	Similarly we have 
	\begin{equation}
	\label{}
	R_2 \lesssim \left(\frac{1}{\ell}\right)^{\ell-1}\Vert h_{in}\Vert_{H^{\sigma}_m}
	\end{equation}
	for $\ell\ge n_0$.
	Again by \eqref{pole:dis} and \eqref{der:L}, $R_3$ is bounded as
	\begin{equation}
	\label{}
	R_3 \lesssim \abs{g_{\ell,k}}.
	\end{equation}
	The lemma is proved by inserting the above estimate into \eqref{res:Rbd}.
\end{proof}

\begin{lemma}[High frequency]
	\label{big}
	There exists an $a_0>0$ sufficiently large such that the bound \eqref{est:res}  is true for $a> a_0$.
\end{lemma} 
\begin{proof}
	In light of Lemma~\ref{I0I1}, we may choose $a_0$ sufficiently large such that 
	\begin{equation}
	\label{bd:01}
	I_{0}(a) + I_{1}(a) \le  \frac{m\omega_c^2\text{e}^{a}}{10q}.
	\end{equation}
	for all $a\ge a_0$.
	According to the definition of Bernstein modes $b_\ell$, we know that 
	\begin{equation}
	\label{}
	L(ib_\ell\omega_c, k) = 1
	\end{equation} 
	for each $\ell\in\mathbb{N}$.
	By the choice of the constant $a_0$, \eqref{1st:mom},  and the fact
	\begin{equation}
	\label{}
	|b_\ell-n|\ge \frac{1}{2}\ \ \ \ \text{for}\ \ell\neq n,
	\end{equation}
	using \eqref{bd:01}, it follows that
	\begin{align}
	\label{}
	A_k\sum_{\substack{n=1\\ n\neq \ell}}^\infty \left|\frac{2n}{a}\frac{1}{\omega_c}I_n(a) \frac{n}{n^2-b_\ell^2}\right|
	\le \frac{2}{\omega_c}A_k(I_{0}(a)+I_{1}(a))
	\le 
	\frac{1}{5}
	\end{align} 
	for $a\ge a_0$, i.e., $k^2\ge a_0\omega_c^2/2$. 
By choosing $a_0$ large we ensure (using $L(i\omega_c b_\ell,k) = 1$)
 	\begin{equation}
	\label{}
	\frac{6}{5}\ge A_k \left|\frac{2\ell}{a}\frac{1}{\omega_c}I_\ell(a) \frac{\ell}{\ell^2-b_\ell^2}\right|
	\ge \frac{4}{5} ,
	\end{equation}
	for  
	$k^2\ge a_0\omega_c^2/2$ and $\ell\in\mathbb{N}$, which implies
	\begin{align}
	\label{}
	A_k \frac{\ell}{a}\frac{1}{\omega_c}I_\ell(a)\lesssim |b_\ell-\ell|\lesssim A_k \frac{\ell}{a}\frac{1}{\omega_c}I_\ell(a).
	\end{align} 
	Note that the bound \eqref{res:Rbd} is still valid in this case, hence define $R_j$ as therein. 
By \eqref{der:L} and \eqref{ineq:gnkest}, direct calculation gives (with $\alpha$ as in the statement of Lemma \eqref{est:res})
	\begin{align} 
	\label{}
	R_1 =
	\left|\frac{1}{\partial_{y}L(ib_{\ell}\omega_c)}\left(\sum_{|n-\ell|\ge n/2} g_{n,k} \frac{1}{n-b_{\ell}} \right)\right|
	 \lesssim \brak{k}^{-\aa}A_k \frac{1}{a}\frac{1}{\omega_c}I_\ell(a)\Vert h_{in}\Vert_{H^{\sigma}_m}.
	\end{align}
	Note that $A_k \frac{\ell}{a}I_\ell(a) \lesssim 1$. The above estimate implies
	  \begin{align}
	  R_1 \lesssim \brak{k}^{-\aa}\brak{\ell}^{-1}\Vert h_{in}\Vert_{H^{\sigma}_m}.
	  \label{}
	  \end{align}
	Next we estimate $R_2$ as
	\begin{align}
	\label{}
	R_2
	&=
	\left|\frac{1}{\partial_{y}L(ib_{\ell}\omega_c)}\left(\sum_{0<|n-\ell|< n/2} g_{n,k} \frac{1}{n-b_{\ell}} \right)\right|
	\lesssim\brak{k}^{-\aa}\brak{\ell}^{-\beta+1}A_k \frac{\ell}{a}\frac{1}{\omega_c}I_\ell(a)\Vert h_{in}\Vert_{H^{\sigma}_m}
	\nonumber\\&
	\lesssim \brak{k}^{-\aa}\brak{\ell}^{-\beta+1} \Vert h_{in}\Vert_{H^{\sigma}_m}.
	\end{align}
	Note that for $R_2$ to be summable in $\ell$, we only need $\bb > 1$. As in the $a<a_0$ case, we bound $R_3$ as
	\begin{align}
	\label{}
	R_3
	=
	\left|\frac{1}{\partial_{y}L(ib_{\ell}\omega_c)}\left(g_{\ell,k} \frac{1}{\ell-b_{\ell}} \right)\right| 
	\lesssim g_{\ell,k}\Vert h_{in}\Vert_{H^{\sigma}_m}.
	\end{align} 
	This implies the residue at $b_\ell$ is bounded by
	\begin{align}
	\label{}
	\left|r_{\ell,k}\right| \le 
	R_1+R_2+R_3
	\lesssim
	\brak{k}^{-\aa}\ell^{-\gamma}\Vert h_{in}\Vert_{H^{\sigma}_m}
	\end{align}
	where $\gamma = \min{(1, \bb-1)}$, completing the proof.
\end{proof}

\begin{proof}[Proof for Lemma~\ref{Residues}]
Lemma \ref{Residues} follows immediately from Lemmas \ref{small} and \ref{big} by choosing $a_0$ sufficiently large. Note in particular that the constants do not depend on $k$ or $\ell$. 
\end{proof}

Next, we compute the inverse Laplace transform of $\cL [\hat\rho]$ and complete the proof of \eqref{eq:BernExp}. 
The last step is to verify that the only contributions come from the poles at the Bernstein modes. This comes down to verifying certain decay conditions together with summability of the residues. 
\begin{lemma}
	For all $t > 0$, if $\sigma>5/2$ and $m>2$, the following formula holds and the sum is absolutely convergent 
	\begin{align}
    \hat{\rho}(t,k) & = \sum_{\ell = 0}^\infty r_{\ell,k} e^{i b_{\ell} \omega_c t} + r_{-\ell,k} e^{-i b_{\ell} \omega_c t}. 
	\end{align} 
    and the coefficients $r_{\ell,k}$ satisfy the estimates \eqref{ineq:ResEst}. 
\end{lemma}
\begin{proof}
Recall that $\cL[\hat{\rho}](z,k)$ is given by the formula \eqref{eq:Lrho}. 
Note that $\cL[\hat{\rho}](\lambda + i\omega,k)$ is not $L^1_\omega$ for fixed $\lambda > 0$, however, we can make sense of the inverse transform via the Bromwich contour formula  \eqref{eq:invTrans} by extension to $L^2_\omega$.  
Next, we will deform the contour past the infinite number of poles representing the Bernstein modes via the contour represented in Figure \ref{fig}. 
Specifically, we will deform past the imaginary axis to  $\Re z = -M$ and vertically to a distance of $\Im z= \pm ([M]+1/2)\omega_c$, encircling each pole with a standard key-hole contour. See Figure \ref{fig} for a graphical depiction of the contour. 
  \begin{figure}[ht]
        \centering 
        \includegraphics[width= 0.70 \textwidth]{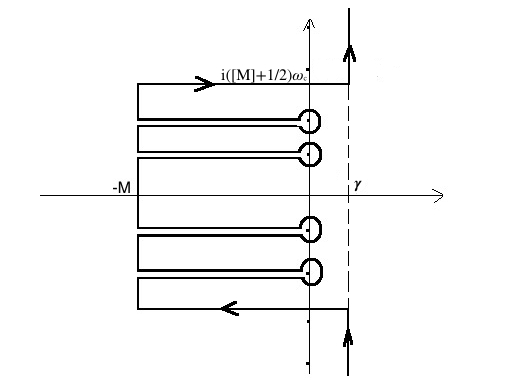}
        \caption{ Pictured above is the contour used to invert the Laplace transform. } \label{fig} 
\end{figure}
The goal is then to pass to the limit $M \rightarrow \infty$ and deduce that the only contribution is given by the residues at the poles. 
This is accomplished provided we prove suitable decay estimates on the contributions of the contour at $\Im z = \pm ([M]+1/2)\omega_c$, $\Re z = -M$, as well as using summability of the residues at the poles. 

    We now begin  by giving the necessary estimates. 
First, observe
	\begin{equation}
	\label{}
	|L(z,k)|=\left|A_k\sum_1^\infty \frac{n}{a}I_n(a)\left(\frac{1}{z-in\omega_c}-\frac{1}{z+in\omega_c}\right)\right|
	\lesssim	\frac{1}{\abs{\Re z}},
	\end{equation}
	which implies that there exists an $M>0$ sufficiently large such that
	\begin{equation}
	\label{}
	\abs{1-L(z,k)} = \left|1+iA_k\sum_1^\infty \frac{n}{a}I_n(a)\left(\frac{1}{z-in\omega_c}-\frac{1}{z+in\omega_c}\right)\right|
	\ge
	\frac{1}{2}
	\end{equation}
	when $|\textup{Re}z|\ge M$.
	On the other hand, from \eqref{ineq:gnkest} we have 
	\begin{equation}
	\label{}
	\left|\sum_{n \in \Integers} g_{n,k} \frac{1}{in\omega_c-z}\right|
	\lesssim	\frac{\norm{h_{in}}_{H^\sigma_m}}{|\Re z|}. 
	\end{equation}
     Therefore, from \eqref{eq:Lrho} we arrive at 
	\begin{equation}
	\label{}
	|\mathcal{L}\hat{\rho}(z,k) | \lesssim \frac{\norm{h_{in}}_{H^\sigma_m}}{|\Re z|}
	\end{equation}
	for $|\Re z| \geq M$. 
	Therefore, we obtain for any $t>0$ (denoting $[M]$ the integer part of $M$),
	\begin{align}
	\label{}
	\left|\frac{1}{2\pi i}\int_{-M-i([M]+1/2)\omega_c}^{-M+i([M]+1/2)\omega_c} \text{e}^{zt} \cL [\rho](z,k)\, dz\right| \lesssim  \norm{h_{in}}_{H^\sigma_m} \text{e}^{-Mt}. 
	\end{align}
    Next, we obtain a good estimate of $\mathcal{L}\hat{\rho}(z,k)$ when $|\Im z|$ is large, which reduces to the case where $\Im z$ is large due to symmetry.  
    In fact, by an easy variation of estimates \eqref{bd:s1}, \eqref{bd:s2}, \eqref{bd:s3}, and \eqref{bd:In2}, if we choose 
	\[\Im z = (n_0 +1/2)\omega_c,\] and $\Re z \in [-M, 1]$ then we obtain 
	\begin{equation}
	\label{}
	|L(z,k)| \leq \frac{1}{2}
	\end{equation}
	for all sufficiently large $n_0$. 
    For the numerator of $\cL[\hat{\rho}]$, we have by \eqref{ineq:gnkest}, 
	\begin{align}
	\label{}
	\left|\sum_{n \in \Integer}  \frac{g_{n,k}}{in\omega_c-z}\right|
	&\lesssim  \left|\sum_{|n-n_0|\ge n/2}  \frac{g_{n,k}}{\omega_c n-\textup{Im}z} \right| + \left|\sum_{0<|n-n_0|< n/2} \frac{g_{n,k}}{\omega_c n-\textup{Im}z} \right| 
	\nonumber\\&
	\lesssim \norm{h_{in}}_{H^\sigma_m} \brak{k}^{-\aa}(n_0^{-1}+\brak{n_0}^{-\beta+1}).
	\end{align}
	Therefore,  $\mathcal{L}\hat{\rho}(z,k)$ converges to $0$ as $n_0\rightarrow\infty$ if $\bb>1$, which implies
	\begin{align}
	\label{}
	\left|\frac{1}{2\pi i}\int_{-M+i([M]+1/2)\omega_c}^{M+i([M]+1/2)\omega_c} \text{e}^{zt} \cL [\rho](z,k)\, dz\right| \lesssim \norm{h_{in}}_{H^\sigma_m} \brak{k}^{-\aa}(M^{-1}+\brak{M}^{-\beta+1}).
	\end{align}
    Finally, we observe that the $r_{\ell,k}$'s are absolutely summable from \ref{est:res}. Hence, we send $M \rightarrow \infty$ and obtain the decomposition \ref{eq:BernExp}. 
\end{proof}

\subsection{Landau damping modes}
Next, we study $\cL[K]$ in the case $k_z \neq 0$ and deduce the Landau damping estimate \eqref{ineq:LDdmpThm}. 
First, we apply \eqref{Besid0} to rewrite $K(t, k)$ as
\begin{align}
\label{kernel}
K(t,k)&=
-\frac{q}{m} \widehat{W}(k)\left(\frac{k_{\perp}^2}{\omega_c}\sin(\omega_ct)+k_3^2t\right)\exp{\left(-2\frac{k_{\perp}^2}{\omega_c^2}\right)} \exp{\left(2\frac{k_{\perp}^2}{\omega_c^2}\cos(\omega_ct)\right)} \widehat{f_3^0}(k_3 t) \nn\\
&=
-\frac{q}{m} \widehat{W}(k)\exp(-a)\sum_{n=-\infty}^{+\infty}
\left(\frac{k_{\perp}^2}{\omega_c}\frac{e^{i\omega_c t-e^{-i\omega_c t}}}{2i}+k_3^2t
\right)I_n(a)e^{in\omega_c t} \widehat{f_3^0}(k_3 t)
\nonumber\\&
=
-\frac{q}{m} \widehat{W}(k)\exp(-a)\sum_{n=-\infty}^{+\infty}
\left(\frac{k_{\perp}^2}{2i\omega_c}(I_{n-1}(a)-I_{n+1}(a))e^{in\omega_c t}+I_n(a)k_3^2te^{in\omega_c t}\right)
\widehat{f_3^0}(k_3 t)
\nonumber\\&
=-\frac{q}{m} \widehat{W}(k)\exp(-a)\sum_{n=-\infty}^{+\infty}
\left(\frac{k_{\perp}^2}{i\omega_c}\frac{n}{a}I_{n}(a)e^{in\omega_c t}+I_n(a)k_3^2te^{in\omega_c t}\right)
\widehat{f_3^0}(k_3t) 
\nonumber\\&
=-\frac{q}{m} \widehat{W}(k)\exp(-a)\sum_{n=-\infty}^{+\infty}
\left(\frac{n\omega_c}{2i}I_{n}(a)e^{in\omega_c t}+I_n(a)k_3^2te^{in\omega_c t}\right)
\widehat{f_3^0}(k_3 t).
\end{align}
Therefore, the Laplace transform satisfies
\begin{align}
\mathcal{L}[K](z, k)& =\frac{q}{m} \widehat{W}(k)\exp(-a) \nonumber \\ & \qquad  \sum_{n \in \Integer} \left( \frac{n\omega_c}{2i} I_n(a) \cL[\widehat{f_3^0}](z - in\omega_c)   
+ I_n(a) \cL[k_3^2 t \widehat{f_3^0}](z - in\omega_c)\right).  \label{eq:cLncol}
\end{align}

\begin{lemma} \label{lem:HsKernelbds}
There holds the following (with constant independent of $k$) for $j \leq \sigma$. 
\begin{align}
\sup_{k: k_3 \neq 0} \sup_{\lambda \geq 0} \abs{\partial_\omega^j \cL[K](\lambda + i\omega,k)} & \lesssim_b 1 + \abs{T_{||} - 1} + \norm{\tilde{f}^0}_{H^{s'}_m}. 
\end{align}
\end{lemma}
\begin{proof}
Recalling \eqref{source} and definition \eqref{def:a}, we see that the lemma follows from the corresponding estimates on $\cL[\widehat{f^0_3}]$ and $\cL[t\widehat{f^0_3}]$.
These are proved as in e.g. \cite{BMM16}. 
\end{proof}

The more non-trivial lemma is the following. 
\begin{lemma} \label{lem:MdEst}
	There exists a $\lambda' \geq 0$ and a $\kappa > 0$ such that 
	\begin{align}
	\inf_{k: k_3 \neq 0} \inf_{\textup{Re} z \geq -\lambda'}\abs{1 - \cL[K](z,k)} \geq \kappa. \label{ineq:MdEstIneq}
	\end{align} 
\end{lemma}

See \cite{BMM16} for a proof that Lemmas \ref{lem:HsKernelbds} and \ref{lem:MdEst} imply the Landau damping estimate \eqref{ineq:LDdmpThm}.

We will break the proof of Lemma \ref{lem:MdEst} into several steps. 
First, we only need to consider the case $k_{\perp} \neq 0$, as the other case is covered by previous works. 
\begin{lemma}
Lemma \ref{lem:MdEst} holds for $k_{\perp} = 0$. 
\end{lemma}
\begin{proof}
If $k_{\perp} = 0$, then $K$ reduces to the one-dimensional unmagnetized problem, studied in e.g. \cite{Penrose,MouhotVillani11,BMM13}.
Using techniques found therein, it is straightforward to verify that for $\delta_0$ sufficiently small, condition \eqref{ineq:MdEstIneq} holds. 
\end{proof}

The next lemmas show that we can restrict ourselves to a compact set in $k_3$ and $\omega$. 
\begin{lemma} \label{lem:clessImz}
There holds the estimate
\begin{align}
\sup_{k: k_3 \neq 0}  \sup_{z: \Re z \geq 0}\abs{\cL[K](z,k)} \lesssim \frac{1}{\brak{z}}. 
\end{align} 
\end{lemma}
\begin{proof}
By integration by parts and $k_3 \neq 0$ (to obtain decay for as $t \rightarrow \infty$), the following holds for all $\lambda > -k_3$: 
\begin{align}
\cL[K](\lambda + i\omega,k) 
& = -\frac{1}{\lambda + i\omega} K(0,k) + \frac{1}{\lambda + i\omega}\int_0^\infty e^{(-\lambda - i \omega)t} \partial_t K(t,k) dt. \label{eq:LKcol} 
\end{align}
First, observe that
\begin{align}
K(0,k)& = 0. 
\end{align}
and 
\begin{align}
\partial_t K & 
= -\frac{q}{m} \widehat{W}(k)(k_{\perp}^2 \cos(\omega_ct)+k_3^2)\exp{\left(-2\frac{k_{\perp}^2}{\omega_c^2}\right)} \exp{\left(2\frac{k_{\perp}^2}{\omega_c^2}\cos(\omega_ct)\right)} \widehat{f_3^0}(k_3 t) \\ 
& \quad +\frac{q}{m} \widehat{W}(k)\left(\frac{k_{\perp}^2}{\omega_c}\sin(\omega_ct)+k_3^2t\right)\exp{\left(-2\frac{k_{\perp}^2}{\omega_c^2}\right)}       
\left(2\frac{k_{\perp}^2}{\omega_c}\sin(\omega_ct)\right) \exp{\left(2\frac{k_{\perp}^2}{\omega_c^2}\cos(\omega_ct)\right)} \widehat{f_3^0}(k_3 t) \nn\\
& \quad - \frac{q}{m} \widehat{W}(k)\left(\frac{k_{\perp}^2}{\omega_c}\sin(\omega_ct)+k_3^2t\right)\exp{\left(-2\frac{k_{\perp}^2}{\omega_c^2}\right)}       
\exp{\left(2\frac{k_{\perp}^2}{\omega_c^2}\cos(\omega_ct)\right)} k_3 (\widehat{f_3^0})'(k_3 t). 
\end{align}
This is estimated via (using $\widehat{W}(k) \lesssim \abs{k}^{-2} \lesssim 1$)
\begin{align}
\abs{\partial_t K(t,k)} \lesssim \abs{\widehat{f_3^0}(k_3 t)} + k_3^2 t\abs{\widehat{f_3^0}(k_3 t)} + \abs{ k_3^2 t  (\widehat{f_3^0})'(k_3 t)}.  
\end{align}
Hence, the integral in \eqref{eq:LKcol} is absolutely integrable uniformly in $k$. 
\end{proof}

\begin{proof}[\textbf{Proof of Lemma \ref{lem:MdEst}}] 
The proof is based on combining \eqref{eq:cLncol} with arguments similar to those appearing in the Penrose criterion \cite{Penrose} (see also \cite{MouhotVillani11,BMM13}). 
First, by arguments contained therein, it is straightforward to deduce the decay estimate 
\begin{align}
\sup_{k:k_3 \neq 0} \sup_{\Re z \geq 0} \abs{L(z,k)} \lesssim \frac{1}{\abs{k_3}}. \label{ineq:kdec}
\end{align}
This ensures that we only need to be concerned with a compact set in $k_3$. 
Analogous to the proof of the Penrose criterion, taking the Laplace transform of $\widehat{f}_3^0$ gives
\begin{align}
\label{}
&\mathcal{L}\left[\widehat{f_3^0}\right](\lambda+i\omega, k)
=
\int_0^\infty e^{-(\lambda+i\omega)t}\int_{-\infty}^{+\infty}e^{-ik_3tv_3}f^0_3(v_3)\,dv_3dt
\nonumber\\&
\qquad=
\int_0^\infty\int_{-\infty}^{+\infty} e^{-(\lambda+i\omega+ik_3v_3)t}f^0_3(v_3)\,dv_3dt
\nonumber\\&
\qquad=
\int_{-\infty}^{+\infty} \frac{1}{\lambda+i\omega+ik_3v_3}f^0_3(v_3)\,dv_3=
\int_{-\infty}^{+\infty} \frac{\lambda-i\omega-ik_3v_3}{\lambda^2+(\omega+k_3v_3)^2}f^0_3(v_3)\,dv_3
\nonumber\\&
\qquad=
\frac{1}{k_3}\int_{-\infty}^{+\infty} \frac{\lambda/k_3}{(\lambda/k_3)^2+(\omega/k_3+v_3)^2}f^0_3(v_3)\,dv_3
\nonumber\\&
\qquad\quad
-
\frac{i}{k_3}\int_{-\infty}^{+\infty} \frac{\omega/k_3+v_3}{(\lambda/k_3)^2+(\omega/k_3+v_3)^2}f^0_3(v_3)\,dv_3, 
\end{align}
hence, by sending $\lambda\rightarrow0$, we obtain
\begin{align}
\label{}
\mathcal{L}\left[\widehat{f_3^0}\right](i\omega, k)
=
\frac{\pi}{k_3}f^0_3\left(-\frac{\omega}{k_3}\right)
-
\frac{i}{k_3}\text{p.v.}\int_{-\infty}^{+\infty} \frac{1}{\omega/k_3+v_3}f^0_3(v_3)\,dv_3.
\end{align}
Similarly, we have
\begin{align}
\label{}
\mathcal{L}\left[ k_3^2t \widehat{f_3^0}(k_3 t)\right](i\omega, k)
=
\text{p.v.}\int_{-\infty}^{+\infty} \frac{1}{\omega/k_3+v_3}(f^0_3)'(v_3)\,dv_3+
i \pi (f^0_3)'\left(-\frac{\omega}{k_3}\right).
\end{align}
Combining this calculation with \eqref{eq:cLncol}, we find that $L(i\omega,k)$ for $\omega \in \Real$ is given by the formula
\begin{align*}
\mathcal{L}(K)(i\omega, k) &  = -\frac{q}{m} \widehat{W}(k)\exp(-a)
\nonumber\\&
\qquad
\left(\sum_{n=-\infty}^{+\infty}
\frac{n\omega_c}{2i}I_{n}(a)\left(\frac{\pi}{k_3}f^0_3\left(-\frac{\omega-n\omega_c}{k_3}\right)
-
\frac{i}{k_3}\text{p.v.}\int_{-\infty}^{+\infty} \frac{1}{(\omega-n\omega_c)/k_3+v_3}f^0_3(v_3)\,dv_3\right)\right. \\ 
& \qquad \left. + \sum_{n=-\infty}^{+\infty}
I_n(a) \left(\text{p.v.}\int_{-\infty}^{+\infty} \frac{1}{(\omega-n\omega_c)/k_3+v_3}(f^0_3)'(v_3)\,dv_3+
i\pi (f^0_3)'\left(-\frac{(\omega-n\omega_c)}{k_3}\right)\right)
\right). 
\end{align*}
Next, recall that 
\begin{align}
f^0_3\left(-\frac{\omega-n\omega_c}{k_3}\right) & = \frac{1}{(2\pi T_{||})^{1/2}}\exp\left(-\frac{1}{2T_{||}} \left(\frac{\omega-n\omega_c}{k_3}\right)^2 \right) + \tilde{f^0_3}\left(\left(\frac{\omega-n\omega_c}{k_3}\right)^2\right)    \\ 
(f^0_3)'\left(-\frac{\omega-n\omega_c}{k_3}\right) & = -\frac{\omega - n \omega_c}{ k_3 T_{||}  (2\pi T_{||})^{1/2}}\exp\left(-\frac{1}{2T_{||}} \left(\frac{\omega-n\omega_c}{k_3}\right)^2 \right) + \frac{2(\omega-n\omega_c)}{k_3} \tilde{f^0_3}'\left(\left(\frac{\omega-n\omega_c}{k_3}\right)^2\right)
\end{align}
Therefore, the imaginary part of $L$ satisfies 
\begin{align}
\frac{m}{-q\pi\widehat{W}(k) \exp(-a) }\Im L(z,k) 
& = \sum_{n=-\infty}^{+\infty} I_n(a) \left( \left(\omega - \frac{n\omega_c}{k_3}\left(1 - \frac{1}{T_{||}}\right)\right) \frac{1}{(2\pi T_{||})^{1/2}}\exp\left(-\frac{1}{2T_{||}} \left(\frac{\omega-n\omega_c}{k_3}\right)^2 \right) \right) \\ 
& \quad + \sum_{n=-\infty}^{+\infty} I_n(a) \left((\tilde{f^0_3})'\left(-\frac{(\omega-n\omega_c)}{k_3}\right) -
\frac{n\omega_c}{2} \frac{1}{k_3}\tilde{f^0_3}\left( \left(\frac{\omega-n\omega_c}{k_3}\right)^2\right)\right) \\ 
& = \sum_{n=-\infty}^{+\infty} I_n(a)  \frac{\omega}{(2\pi T_{||})^{1/2}}\exp\left(-\frac{1}{2T_{||}} \left(\frac{\omega-n\omega_c}{k_3}\right)^2 \right)  + \cQ_{\Im}(\omega,k), 
\end{align}
where 
\begin{align}
\cQ_{\Im}(\omega,k) & = -\sum_{n=-\infty}^{+\infty} I_n(a) \left( \left( \frac{n\omega_c}{k_3}\left(1 - \frac{1}{T_{||}}\right)\right) \frac{1}{(2\pi T_{||})^{1/2}}\exp\left(-\frac{1}{2T_{||}} \left(\frac{\omega-n\omega_c}{k_3}\right)^2 \right) \right) \\ 
& \quad + \sum_{n=-\infty}^{+\infty} I_n(a) \left((\tilde{f^0_3})'\left(\left(\frac{(\omega-n\omega_c)}{k_3}\right)^{2} \right) -
\frac{n\omega_c}{2} \frac{1}{k_3}\tilde{f^0_3}\left( \left(\frac{\omega-n\omega_c}{k_3}\right)^2\right)\right). 
\end{align}
Observe that for all $\eps > 0$, there exists a $\delta > 0$ such that 
\begin{align}
\inf_{k_3 \in \Integer_\ast : \abs{k_3} \leq \eps^{-1}} \inf_{\eps < \abs{\omega} < \eps^{-1}} \abs{\sum_{n=-\infty}^{+\infty} I_n(a)  \frac{\omega}{(2\pi T_{||})^{1/2}}\exp\left(-\frac{1}{2T_{||}} \left(\frac{\omega-n\omega_c}{k_3}\right)^2 \right)} > \delta. \label{eq:ILbig}
\end{align}
Next, observe that the real part of $L$ satisfies 
\begin{align*}
\Re L(z,k) & = -\frac{q}{m} \widehat{W}(k)\exp(-a)
\nonumber\\
& \qquad
\left(\sum_{n=-\infty}^{+\infty}
\frac{n\omega_c}{2}I_{n}(a)\left(-
\frac{1}{k_3}\text{p.v.}\int_{-\infty}^{+\infty} \frac{1}{(\omega-n\omega_c)/k_3+v_3}f^0_3(v_3)\,dv_3\right)\right. \\ 
& \qquad \left. + \sum_{n=-\infty}^{+\infty}
I_n(a) \left(\text{p.v.}\int_{-\infty}^{+\infty} \frac{1}{(\omega-n\omega_c)/k_3+v_3}(f^0_3)'(v_3)\,dv_3\right)
\right) \\ 
& = -\frac{q}{m} \widehat{W}(k)\exp(-a) \nonumber\\
& \quad \Bigg(\sum_{n=-\infty}^{+\infty} I_{n}(a) p.v. \int_{-\infty}^{\infty} \frac{1}{2} \frac{v_3 - n\omega_c/k_3}{(\omega-n\omega_c)/k_3+v_3} \frac{1}{(4\pi T_{||})^{1/2}} \exp\left( - \frac{v_3^2}{4 T_{||}}\right) dv_3 \Bigg) + \cQ_{\Re}(\omega,k), 
\end{align*}
where 
\begin{align*}
\cQ_{\Re}(\omega,k) & = -\frac{q}{m} \widehat{W}(k)\exp(-a)
\nonumber\\
& \qquad
\left(\sum_{n=-\infty}^{+\infty}
\frac{n\omega_c}{2}I_{n}(a)\left(-
\frac{1}{k_3}\text{p.v.}\int_{-\infty}^{+\infty} \frac{1}{(\omega-n\omega_c)/k_3+v_3}\tilde{f^0}_3(v^2_3)\,dv_3\right)\right. \\ 
& \qquad \left. + \sum_{n=-\infty}^{+\infty}
I_n(a) \left(\text{p.v.}\int_{-\infty}^{+\infty} \frac{1}{(\omega-n\omega_c)/k_3+v_3}(\tilde{f^0}_3)'(v^2_3)\,dv_3\right)
\right) \\ 
& +\frac{q}{m} \widehat{W}(k)\exp(-a) \nonumber \\ 
& \qquad \left(\sum_{n=-\infty}^{+\infty} I_n(a) \left(\text{p.v.}\int_{-\infty}^{+\infty} \frac{1}{2}\frac{v_3}{(\omega-n\omega_c)/k_3+v_3}\left(1 - \frac{1}{T_{||}}\right) \frac{1}{(4\pi T_{||})^{1/2}} \exp\left( - \frac{v_3^2}{4 T_{||}}\right)   \,dv_3\right)\right). 
\end{align*}
Then note that 
\begin{align}
-\frac{q}{m} \widehat{W}(k)\exp(-a) \Bigg(\sum_{n=-\infty}^{+\infty} p.v. \int_{-\infty}^{\infty} \frac{1}{2} \frac{v_3 - n\omega_c/k_3}{(-n\omega_c)/k_3+v_3} \frac{1}{(4\pi T_{||})^{1/2}} \exp\left( - \frac{v_3^2}{4 T_{||}}\right) dv_3 \Bigg) < 0. \label{eq:RLneg} 
\end{align}
First, let us argue in the case $\tilde{f_3^0} = 0$ and $T_{||} = 1$. 
In this case, $\cQ_{\Im} = \cQ_{\Re} = 0$. 
The calculations on $\Re L$ together with \eqref{eq:RLneg} show that $\Re L(0,k) < 0$. Further, it is clear that $\Re L(i\omega,k)$ is a smooth function in $\omega$ and hence there is an $\eps > 0$ such that for $\abs{\omega} < \eps$, there holds $\Re L(i\omega,k) \leq 0$ and so in this region, $\abs{1-L(i\omega,k)} \geq 1$. 
For $\omega$ away from zero, the imaginary part is bounded strictly away from zero by \eqref{eq:ILbig}, and hence, the lower bound follows. 
The general case follows by 
observing (recall \eqref{def:a} and properties of $I_n$ from Appendix \ref{sec:Bessel}) that the following estimate holds, 
\begin{align}
\abs{\cQ_{\Im}} + \abs{\cQ_{\Re}} & \lesssim \left(\abs{1 - \frac{1}{T_{||}}} + \norm{f^3_0}_{H^{s'}_m} \right) \abs{k_{\perp}}^2 \widehat{W}(k) \exp(-a) \sum_{n=1}^\infty I_n(a) \frac{n}{a k_3}  \\ 
& \lesssim \abs{1 - \frac{1}{T_{||}}} + \norm{f^3_0}_{H^{s'}_m}. 
\end{align} 
We have now verified that (recall \eqref{ineq:kdec} deals with large $k_3$ and Lemma \ref{lem:clessImz} with large $\omega$) 
\begin{align}
\inf_{k\in\Integer^3: k_3 \neq 0} \inf_{\omega} \abs{1- \cL[K](i\omega,k)} \geq \kappa. 
\end{align}
By the holomorphy of $L(z,k)$ for $\Re z \geq 0$ the decay of $L(\lambda + i\omega,k)$ at large $\omega$ from Lemma \ref{lem:clessImz}, it follows that there exists $\lambda',\kappa > 0$ such that 
\begin{align}
\inf_{k\in\Integer^3_\ast: k_3 \neq 0} \inf_{\textup{Re} z \in [0,\lambda']} \abs{1 - L(z,k)} \geq \kappa/2.  
\end{align}
As in \cite{Penrose} (see also \cite{BMM13}), we now extend to all $\Re z \geq \lambda'$ using the argument principle. Indeed, as $K(z,k)$ is holomorphic for $\Re z \geq 0$ and the value $1$ is not taken on $i\Real$, $L$ can only take the value one in the right half-plane if the curve $ \omega \mapsto L(i\omega,k)$ has a positive winding number around one. However, this is impossible as $L$ is vanishingly small at large $\omega$ by Lemma \ref{lem:clessImz} and large $k_3$ by \eqref{ineq:kdec}, the imaginary part of $L(i\omega,k)$ is non-vanishing for $\omega$ away from a small neighborhood of zero by \eqref{eq:ILbig}, and $L(i\omega,k) \leq 0$ for $\omega$ sufficiently small by \eqref{eq:RLneg} (we have also used that $T_{||} \neq 1$ and $\norm{\tilde{f^0}} \ll 1$). Hence, Lemma \ref{lem:MdEst} follows. 
\end{proof}

\section{Collisional case}
\subsection{Collisional relaxation via energy methods} \label{sec:Energy}
In this section we use the energy method of Yan Guo \cite{Guo02,Guo03,Guo06,Guo12} to prove \eqref{ineq:coldec}. 
First, we recall the basic energy structure of the equation. 
It is convenient to denote the collision operator 
\begin{align}
Lf = \grad_v \cdot \left(\grad_v f + v f \right). 
\end{align}
Denote the natural Gaussian weighted space and inner product
\begin{subequations}
\begin{align}
\brak{f,g}_\mu & = \int f(x,v)g(x,v) \frac{1}{\mu(v)} dv dx \\ 
\norm{f}^2_{L^2_\mu} & = \brak{f,f}_\mu = \int \frac{\abs{f(x,v)}^2}{\mu(v)} dv dx. 
\end{align}
\end{subequations}
As it simplifies the energy structure a little, for the remainder section let us consider the case $q = m = 1$; the general case is an easy variant. 
Recall the two relevant hydrodynamic quantities (the density and momentum) and the orthogonal complement: 
\begin{subequations}
	\begin{align}
	\rho(t,x) & = \int h(t,x,v) dv \\ 
	\rho u(t,x) & = \int v h(t,x,v) dv \\ 
	g(t,x,v) & = h(t,x,v) - \rho(t,x) \mu(v), 
	\end{align}
\end{subequations}
and note that the momentum depends only on $g$ 
\begin{align}
\rho u(t,x) = \int v g(t,x,v) dv, \label{eq:rhoug}
\end{align}
and that $g$ is average zero at each $x$
\begin{align}
\int g(t,x,v) dv = 0.  \label{eq:gorth}
\end{align}
Next, define the following natural energy (the quadratic variation of the Boltzmann entropy $+$ electric field energy), which is a Lyapunov functional for the dynamics\footnote{ According to Bernstein \cite{Bernstein58} this was first observed by Newcomb.}, 
\begin{align}
\cE_0(t) = \frac{1}{2}\norm{h}_{L^2_\mu}^2 + \frac{1}{2}\norm{E}_{L^2}^2. 
\end{align} 
Specifically, we have the following H-theorem.
 Note that since $L$ has a non-trivial kernel, this does not immediately imply relaxation to global equilibrium. 
\begin{proposition}[H-theorem] \label{prop:Hthm}
	There holds the following energy balance for strong solutions of \eqref{def:VPElin} with $h_{in} \in L^2_\mu$ (with $f^0 = \mu$): 
	\begin{align}
	\frac{d}{dt}\cE_0(t) = -\nu\brak{h,Lh}_{\mu}.  \label{eq:Hthm}
	\end{align}
	Moreover, there exists a $\lambda > 0$ such that 
	\begin{align}
	\brak{h,Lh}_{\mu} = \brak{g,Lg}_{\mu} \geq \lambda\norm{g}_{L^2_\mu}^2. \label{ineq:notCoerc}
	\end{align}
\end{proposition}
\begin{remark} 
	One can be more general in the case $\nu = 0$. Indeed, if $\nu = 0$ and $f^0(v) = m(\abs{v}^2)$ for $m$ strictly monotone decreasing, then the following energy is conserved
\begin{align}
\cE = \frac{1}{2}\int \frac{\abs{h(t,x,v)}^2}{-m'(\abs{v}^2)} dv dx + \frac{1}{2}\norm{E(t)}_{L^2}^2. 
\end{align} 
\end{remark}
\begin{proof}
	 To see the energy/entropy dissipation balance \eqref{eq:Hthm}, consider the following (denoting $\phi = \Delta^{-1} \rho$), 
	 \begin{align}
	 \frac{d}{dt}\cE_0(t) & = \int \frac{h}{\mu} \left(-v\cdot \grad_x h  - (v\times B) \cdot \grad_v h - E \cdot \grad_v \mu + \nu L h\right) dv dx \\ 
	 & \quad - 2\int \phi \left(-v\cdot \grad_x h  - (v\times B) \cdot \grad_v h - E \cdot \grad_v \mu + \nu Lh\right) dx dv.
	 \end{align}
	 Then, \eqref{eq:Hthm} follows from the following (using $\mu = e^{-v^2/2}/(2\pi)^{3/2}$), 
	 \begin{align}
	 \int \frac{h}{\mu} \left(-v\cdot \grad_x h\right) dx dv & = 0 \\
	 \int \frac{h}{\mu} (v\times B) \cdot \grad_v h dx dv & = b \int \frac{h}{\mu}\left(v_1 \partial_{v_2} h - v_2 \partial_{v_1} h \right) dx dv  = 0 \\ 
	 \int \frac{h}{\mu} \left(- E \cdot \grad_v \mu \right) dv dx & = \int h E \cdot v dx dv \\ 
	 \int \phi \left(-v\cdot \grad_x h\right) dx dv & = - \int h E \cdot v dx dv \\
	 \int \phi \left(- (v\times B) \cdot \grad_v h\right) dx dv & = b\int \phi \left(v_1 \partial_{v_2} h - v_2 \partial_{v_1} h \right) dx dv = 0 \\ 
	 \int \phi \left( - E \cdot \grad_v \mu \right) dx dv & = 0 \\
	 \int \phi \left(\nu Lh\right) dx dv & = 0. 
	 \end{align}
	 
	 To see \eqref{ineq:notCoerc}, first note that, because $L\mu = 0$, there holds
	 \begin{align*}
	 \brak{h,Lh}_{\mu} = \brak{\rho \mu + g, L(\rho \mu +g)}_{\mu} = \int_{\TT^3} \rho(t,x) \left(\int_{\Real^3} L(g)(t,x,v) dv\right) dx + \brak{g,Lg}_{\mu} = \brak{g,Lg}_{\mu}. 
	 \end{align*}
	 Then, \eqref{ineq:notCoerc} follows from the classical spectral gap of the Fokker-Planck operator in the Maxwellian weighted space (see e.g. \cite{GallayWayne02} and the references therein) together with the orthogonality condition \eqref{eq:gorth}. 
\end{proof}

The important point here is that the H-theorem \eqref{eq:Hthm} does not yield a coercive estimate in $L^2_\mu$ due to the non-trivial null-space of the collision operator $L$. 
It is natural to apply hypocoercivity, however, due to the presence of the non-local term and the magnetic field, it is not clear that the standardized hypocoercivity algorithm, described in e.g. \cite{Villani2009} and the references therein, can be applied.   
However, the energy method of Yan Guo, devised for dealing with non-local and nonlinear collision operators in \cite{Guo02,Guo03,Guo06,Guo12}, can be adapted here. 
The method is based on using the hydrodynamic equations to provide coercivity up to oscillatory factors, which essentially allows one to conclude that the dynamics stay away from the kernel of the collision operator on average.    
For linear problems, we remark that the method indeed reduces to a variant of hypocoercivity. 

Denote 
\begin{align}
u^{\perp} := \begin{pmatrix}
-u^2 \\ 
u^1 \\
0
\end{pmatrix}.
\end{align}
The hydrodynamic equations are (recall we have set $q = m = 1$ for simplicity of presentation in this section), 
\begin{subequations} 
\begin{align}
\partial_t \rho + \grad_x \cdot (\rho u) & = 0 \label{eq:Contin} \\  
\partial_t(\rho u) + \grad_x \cdot \int (v \otimes v) h dv + b\rho u^\perp - E & = -\nu \rho u  \label{eq:momen} \\ 
\partial_t \int \abs{v}^2 h dv + \grad_x \cdot \int v \abs{v}^2 h dv & = -2\nu \int \abs{v}^2 g dv. \label{eq:temp}
\end{align}
\end{subequations}
We will need the following important estimates on the momentum (recall \eqref{eq:rhoug}) 
\begin{subequations} 
	\begin{align}
	\norm{\rho u}_{L^2_x} & \lesssim \norm{g}_{L^2_\mu} \label{ineq:momBd1} \\  
	\norm{\grad_x \cdot (\rho u)}_{L^2_x} & \lesssim \norm{\grad_x g}_{L^2_\mu}. \label{ineq:momBd2}
	\end{align}
\end{subequations}

In Guo's energy method, the main step is an approximate positivity up to a (potentially) time oscillating term. 
The proof makes key use of \eqref{eq:Contin}, \eqref{eq:momen}.  

\begin{lemma}[Positivity up to oscillation] \label{lem:pos}
	Let 
	\begin{align}
	G(t) = \brak{\rho,\grad_x \cdot (\rho u)}_{L^2}.
	\end{align}
	Then, there exists a constant $C > 0$ such that 
	\begin{align}
	6\norm{\grad \rho(t)}_{L^2_x}^2 + \frac{1}{2}\norm{\rho(t)}_{L^2_x}^2 \leq C\norm{\grad_x g}_{L^2_\mu}^2 +  \frac{d}{dt}G(t).
	\end{align}
	Finally, note that by \eqref{ineq:momBd1},\eqref{ineq:momBd2}, there holds
	\begin{align}
	\abs{G(t)} \lesssim \norm{\rho}_{L^2_x} \norm{\grad_x g}_{L^2_\mu}.  
	\end{align}
\end{lemma}
\begin{proof}
	First, notice that (denoting the Kronecker delta as $\delta_{ij}$)
	\begin{align}
\int v_k v_i h dv & = \int v_k v_i \rho \mu(v) + g dv \nonumber \\ 
                 & = \rho \delta_{ik} \int \abs{v}^2 \mu (v) dv + \int v_k v_i g dv \nonumber \\ 
& = 2\rho \delta_{ik} + \int v_k v_i g dv. \label{eq:vivkf}
	\end{align}
	Therefore, taking the divergence of \eqref{eq:momen}, multiplying by $\rho$, integrating, and using that $\grad_x \cdot E = \rho$, we have 
	\begin{align}
	\brak{\rho,\partial_t \grad_x \cdot (\rho u)}_{L^2} + \brak{\rho,b\grad_x \cdot (\rho u^\perp)}_{L^2} - 6\int \abs{\grad \rho}^2 - \int \abs{\rho}^2 & = -\nu \brak{\rho,\grad_x \cdot(\rho u)}_{L^2}.
	\end{align}
	Using \eqref{eq:Contin}  in the time derivative term, we have (note that for general $q$, the electric field term comes with the correct sign regardless of the sign of $q$), 
	\begin{align}
	6\int \abs{\grad \rho}^2 + \int \rho^2 = \frac{d}{dt}\brak{\rho,\grad_x \cdot (\rho u)} + \norm{\grad_x \cdot (\rho u)}_2^2 + \nu \brak{\rho,\grad_x \cdot(\rho u)} + \brak{\rho,b\grad_x \cdot (\rho u^\perp)}. \label{eq:gradRho}
	\end{align}
	Note that 
	\begin{align}
	\rho u^{\perp} = \int v^{\perp} f dv = \int v^{\perp} g dv,
	\end{align}
	and hence 
	\begin{align}
	\norm{\grad_x \cdot \left(\rho u^\perp\right)}_{L^2_x} \lesssim \norm{\grad_x g}_{L^2_\mu}. 
	\end{align}
	Therefore, using \eqref{ineq:momBd1}, \eqref{ineq:momBd2} we have from \eqref{eq:gradRho}, for some $C > 0$,  
	\begin{align}
	6\norm{\grad \rho}_{L^2_x}^2 + \norm{\rho}_{L^2_x}^2 \leq \frac{d}{dt}G(t) + C\norm{\grad_x g}_{L^2_\mu}^2 + C(\nu + b)\norm{\rho}_{L^2_x} \norm{\grad_x g}_{L^2_\mu}. 
	\end{align}
	The lemma then follows (up to re-defining $C$). 
\end{proof}
Using Lemma \ref{lem:pos}, we now complete the proof of the decay estimate \eqref{ineq:coldec}.
Define the following additional ensergy 
\begin{align} 
\mathcal{E}_1 & = \frac{1}{2}\int \frac{\abs{\grad_x h}^2}{\mu} dx dv + \frac{1}{2}\int \abs{\grad_x E}^2 dx.
\end{align}
From Proposition \eqref{prop:Hthm} (and that $\grad_x$ commutes with the \eqref{def:VPElin}), we have
\begin{align}
\frac{d}{dt}\cE_0  & = -\nu \brak{g,Lg}_\mu \leq -\nu \lambda \norm{g}_{L^2_\mu}^2 \\ 
\frac{d}{dt}\cE_1 & = -\nu \brak{\grad_x h, L\grad_x h}_{\mu}. 
\end{align}
Note that the presence of $\grad_x$ does not affect the orthogonality:  
\begin{align}
\brak{\grad_x h,L\grad_x h}_\mu = \brak{\grad_x g,L\grad_x g}_\mu.
\end{align}
Let $\delta > 0$ be a parameter to be chosen small later. Then, by the positivity from Lemma \ref{lem:pos} and the spectral gap \eqref{ineq:notCoerc}, we have for some constant $C$, 
\begin{align}
\frac{d}{dt}\left(\cE_0 + \cE_1\right) & \leq -\nu \lambda \norm{g}_{L^2_\mu}^2  -\nu(1-\delta) \lambda \norm{\grad_x g}_{L^2_\mu}^2 - \nu \delta \lambda \norm{\grad_x g}_{L^2_\mu}^2 \\ 
& \leq -\nu \lambda \norm{g}_{L^2_\mu}^2 - \nu (1-\delta) \lambda \norm{\grad_x g}_{L^2_\mu}^2 - \frac{\nu \delta \lambda}{C}\norm{\grad_x \rho}_{L^2_x}^2 + \frac{\nu \lambda \delta}{C}\frac{d}{dt}G(t). 
\end{align}
Hence, for $\delta'$ small enough (depending only on $C,\lambda, \delta$), 
\begin{align}
\frac{d}{dt}\left(\cE_0 + \cE_1 - \frac{\nu \lambda \delta}{C} G\right) \leq -\nu \delta' \left(\norm{g}_{L^2_\mu}^2 + \norm{\grad_x g}_{L^2_\mu}^2 + \norm{\grad_x \rho}_2^2\right). 
\end{align}
Next, we observe from \eqref{ineq:momBd1}, \eqref{ineq:momBd2} 
\begin{align}
\abs{G(t)} \lesssim \norm{\grad_x g}_{L^2_\mu}^2 + \norm{\rho}_2^2 \lesssim  \cE_0 + \cE_1, \label{ineq:GenBd}
\end{align}
and hence for $\nu \in (0,1)$, that there exists some universal $C_0 > 0$ such that (recall $\int_{\TT^3} \rho dx = 0$), 
\begin{align}
\cE_0 + \cE_1 - \frac{\nu \lambda \delta}{C} G \leq C_0\left(\norm{g}_{L^2_\mu}^2 + \norm{\grad_x g}_{L^2_\mu}^2 + \norm{\grad_x \rho}_2^2\right),
\end{align}
(recall that $\norm{h}_{L^2_\mu}^2 = \norm{g}_{L^2_\mu}^2 + \norm{\rho}_{2}^2$) and 
\begin{align}
\frac{d}{dt}\left(\cE_0 + \cE_1 - \frac{\nu \lambda \delta}{C} G\right) \leq -\nu \frac{\delta}{C_0} \left(\cE_0 + \cE_1 - \frac{\nu \lambda \delta}{C} G\right), 
\end{align}
which implies the exponential decay 
\begin{align}
\left(\cE_0 + \cE_1 - \frac{\nu \lambda \delta}{C} G\right)(t) \leq e^{-\nu \frac{\delta}{C_0} t} \left(\cE_0 + \cE_1 - \frac{\nu \lambda \delta}{C} G\right)(0). 
\end{align} 
Using again \eqref{ineq:GenBd} gives the following for $\delta \nu$ sufficiently small
\begin{align}
\left(\cE_0 + \cE_1\right)(t) \lesssim e^{-\frac{\nu \delta}{C_0} t} \left(\cE_0 + \cE_1\right)(0). 
\end{align}
This completes the proof of the exponential decay estimate \eqref{ineq:coldec} in the case $\sigma = 1$. 
Using that $\grad_x$ derivatives commute with the equation, it is straightforward to extend this estimate to all $\sigma \geq 1$.

\subsection{Landau damping modes}
In this section, we prove the uniform Landau damping and enhanced collisional relaxation described by \eqref{ineq:LDcol} for modes with $k_3 \neq 0$. 

\subsubsection{Volterra equation reduction} \label{sec:VoltColl}
In this section, we derive a Volterra equation for the density. 
This is analogous to calculations done in \cite{Tristani2016,B17}, however, the magnetized case is more technical due to the more complicated characteristics. 
As in these previous works, this uses crucially that the linear Fokker-Planck collision operator is relatively simple under the Fourier transform. 

Taking the Fourier transform of \eqref{def:VPElin} with respect to both $x$ and $v$ gives the following first order equation 
\begin{align}
\partial_t \hat{h} + \left(A \eta - k\right) \cdot \grad_\eta \hat{h} + \hat{E}(t,k) \cdot i \eta \hat{f^0}(\eta) = - \nu \abs{\eta}^2 \hat{h}, 
\end{align}
where 
\begin{align}
A = 
\begin{pmatrix} 
\nu & \omega_c & 0 \\ 
-\omega_c & \nu & 0 \\  
0 & 0 & \nu
\end{pmatrix} 
.
\end{align}
As in the collisionless case, we solve this via the method of characteristics. 
The characteristic emanating from $(k,\eta)$ is given by 
\begin{align}
\bar{\eta}(t;k,\eta) = e^{tA} \eta - \int_0^t e^{(t-\tau)A} k d\tau. 
\end{align}
Note that 
\begin{align}
e^{tA} = e^{\nu t } \begin{pmatrix} 
\cos \omega_c t & \sin \omega_c t & 0 \\ 
-\sin \omega_c t & \cos \omega_c t & 0 \\
0 & 0 & 1
\end{pmatrix}
. 
\end{align}
We define the \emph{profile} $f$ as 
\begin{align}
\hat{h}(t,k,\eta) = \hat{f}\left(t,k, e^{-t A} \eta + \int_0^t e^{-\tau A} k d\tau\right). 
\end{align}
The method of characteristics gives the following evolution equation for $\hat{f}(t,k,\eta)$ as 
\begin{align}
\partial_t \widehat{f} + \widehat{E}(t,k) \cdot i\bar{\eta}(t;k,\eta) \widehat{f^0} (\bar{\eta}(t;k,\eta)) = - \nu \abs{\bar{\eta}(t;k,\eta)}^2 \hat{f}. \label{eq:proff}
\end{align}
Next, define the propagator: 
\begin{align}
S(t,\tau;k,\eta) = \exp\left[-\nu \int_\tau^t \abs{\bar{\eta}(s;k,\eta)}^2 ds\right]. 
\end{align}
Integrating \eqref{eq:proff} gives
\begin{align}
\hat{f}(t,k,\eta) = S(t,0;k,\eta) \widehat{f_{in}} - \int_0^t S(t,\tau;k,\eta) \widehat{E}(\tau,k) \cdot i\bar{\eta}(\tau;k,\eta) \widehat{f^0} (\bar{\eta}(\tau;k,\eta)) d\tau. \label{eq:proffInt}
\end{align}
Define the Orr-critical frequency as (here $CT$ stands for `critical time'): 
\begin{align}
\eta_{CT}(t,k) = \int_0^t e^{-\tau A} k d\tau;
\end{align}
the relevance of this frequency is due to the relation 
\begin{align}
\hat{\rho}(t,k) = \hat{f}(t,k,\eta_{CT}(t,k)). 
\end{align}
Hence, evaluating \eqref{eq:proffInt} at $\eta_{CT}$ gives
\begin{align}
\hat{\rho}(t,k) & = S(t,0;k,\eta_{CT}(t)) \widehat{f_{in}}(k,\eta_{CT}(t)) \\ & \quad - \int_0^t S(t,\tau;k,\eta_{CT}(t)) \widehat{E}(\tau,k) \cdot i\bar{\eta}(\tau;k,\eta_{CT}(t)) \widehat{f^0} (\bar{\eta}(\tau;k,\eta_{CT}(t))) d\tau. \label{eq:rhoetaCT}
\end{align}
Note that
\begin{align}
\bar{\eta}(\tau;k,\eta_{CT}(t)) & = e^{\tau A} \int_0^t e^{-s A} k ds - \int_0^\tau e^{(\tau-s)A} k ds = e^{\tau A} \int_\tau^t e^{-s A} k ds, 
\end{align}
and hence, 
\begin{align}
S(t,\tau;k,\eta_{CT}(t,k)) & = \exp\left[-\nu \int_\tau^t \abs{\bar{\eta}(s;k,\eta_{CT}(t,k))}^2 ds\right] \nonumber \\ & = \exp\left[-\nu \int_\tau^t \abs{ \left( e^{sA} \int_s^t e^{-r A} dr \right) k}^2 ds \right].  \label{eq:SAexp}
\end{align}

It is absolutely crucial that we express $S(t,\tau;\eta_{CT}(t,k)$ as a function of $t-\tau$ in order to ultimately reduce \eqref{eq:rhoetaCT} to a Volterra equation for $\rho$. 
This is the content of the next lemma. 

\begin{lemma}\label{lem:aijTech} 
Define 
\begin{align}
a_{11}(t) & = \frac{\nu}{\omega_c^2 + \nu^2}\left[1 - \cos \omega_c t e^{-\nu t} \right]  - \frac{\omega_c}{\omega_c^2 + \nu^2} e^{-\nu t} \sin \omega_c t  \\ 
a_{12}(t)& = \frac{\omega_c}{\omega_c^2 + \nu^2}\left[1 - \cos \omega_c t e^{-\nu t} \right]  + \frac{\nu}{\omega_c^2 + \nu^2} e^{-\nu t} \sin \omega_c t. 
\end{align}
Then, for all $s \leq t$, there holds
\begin{align}
e^{sA} \int_s^t e^{-r A} dr
= 
\begin{pmatrix} 
a_{11}(t-s) & a_{12}(t-s) &  0\\ 
-a_{12}(t-s) & a_{22}(t-s) & 0 \\ 
0 & 0 & \frac{1}{\nu}\left(1 - e^{-\nu (t-s)}\right). 
\end{pmatrix} 
\end{align}
As a result, $S(t,\tau;k,\eta_{CT}(t,k))$ can be written as a function of only $t-\tau$ and $k$, which we define as follows: 
\begin{align}
S(t-\tau,k) := S(t,\tau;k,\eta_{CT}(t,k)),
\end{align}
where 
\begin{align}
S(t,k) & := \exp\left[-\nu^{-1} k_3^2 \left( t + 2 \frac{e^{-\nu t} - 1}{\nu} - \frac{e^{-2\nu t}- 1 }{2\nu}\right)\right] \\ 
& \quad \times \exp \Bigg[ - \frac{\nu \abs{k_{\perp}}^2}{\nu^2 + \omega_c^2} \Bigg( t  -  \frac{2\nu}{\nu^2 + \omega^2_c} + \frac{2\nu}{\nu^2 + \omega_c^2} e^{-\nu t} \cos \omega_{c}t \\ & \quad\quad - \frac{2\omega_c}{\nu^2 + \omega_c^2} \sin\omega_c t e^{-\nu t}  - \frac{e^{-2\nu t}- 1 }{2\nu}      \Bigg)  \Bigg] \label{def:S} \\ 
& = \exp\left[-\nu \int_\tau^t\abs{k_3}^2 \left(\frac{1 - e^{-\nu (t-s)}}{\nu}\right)^2 ds \right] \nonumber \\ & \quad \times \exp\left[ -\nu \frac{\abs{k_{\perp}}^2}{\nu^2 + \omega_c^2} \int_\tau^t \left[1 - \cos \omega_c (t-s) e^{-\nu (t-s)} \right]^2 + e^{-2\nu(t-s)} \sin^2 \omega_c (t-s) ds \right]
\end{align}
Similarly, 
\begin{align}
\bar{\eta}(\tau;k,\eta_{CT}(t,k)) & = 
\begin{pmatrix} 
a_{11}(t-\tau) k_1 + a_{12}(t-\tau) k_2 \\ 
-a_{12}(t-\tau) k_1 + a_{11}(t-\tau) k_2 \\ 
\frac{k_3}{\nu}\left( 1 - e^{-\nu(t-\tau)}\right)
\end{pmatrix} 
\label{eq:barEta}
\end{align}

\end{lemma} 

\begin{proof} 
First, note the following identities
\begin{align}
\int_s^t e^{-\nu r}\cos \omega_c r dr & = \frac{1}{2} \int_s^t e^{(-\nu + i\omega_c) r} + e^{(-\nu + i\omega_c) r} dr \\ 
& \hspace{-3cm} = \frac{(-\nu - i\omega_c)}{2(\nu^2 + \omega_c^2)}\left( e^{(-\nu + i\omega_c) t} - e^{(-\nu + i\omega_c) s}\right) + \frac{(-\nu + i\omega_c)}{2(\nu^2 + \omega_c^2)}\left( e^{(-\nu - i\omega_c) t} - e^{(-\nu - i\omega_c) s}\right) \\ 
& \hspace{-3cm} = \frac{-\nu}{\nu^2 + \omega_c^2}\left( e^{-\nu t}\cos \omega_c t - e^{-\nu s}\cos \omega_c s\right) + \frac{\omega_c}{\nu^2 + \omega_c^2}\left( e^{-\nu t}\sin \omega_c t - e^{-\nu s}\sin \omega_c s\right), 
\end{align}
and 
\begin{align}
\int_s^t e^{-\nu r}\sin \omega_c r dr & = \frac{1}{2i} \int_s^t e^{(-\nu + i\omega_c) r} - e^{(-\nu + i\omega_c) r} dr \\ 
& \hspace{-3cm} = \frac{(-\nu - i\omega_c)}{2i(\nu^2 + \omega_c^2)}\left( e^{(-\nu + i\omega_c) t} - e^{(-\nu + i\omega_c) s}\right) - \frac{(-\nu + i\omega_c)}{2i(\nu^2 + \omega_c^2)}\left( e^{(-\nu - i\omega_c) t} - e^{(-\nu - i\omega_c) s}\right) \\ 
& \hspace{-3cm} = \frac{-\nu}{\nu^2 + \omega_c^2}\left( e^{-\nu t}\sin \omega_c t - e^{-\nu s}\sin \omega_c s\right) - \frac{\omega_c}{\nu^2 + \omega_c^2}\left( e^{-\nu t}\cos \omega_c t - e^{-\nu s}\cos \omega_c s\right). 
\end{align}
This implies, after using the angle sum/difference formulas
\begin{align}
e^{sA} \int_s^t e^{-r A} dr
= 
\begin{pmatrix} 
a_{11}(t-s) & a_{12}(t-s) &  0\\ 
-a_{12}(t-s) & a_{22}(t-s) & 0 \\ 
0 & 0 & \frac{1}{\nu}\left(1 - e^{-\nu (t-s)}\right). 
\end{pmatrix} 
,
\end{align}
where 
\begin{align}
a_{11}(t-s) & = \frac{\nu}{\omega_c^2 + \nu^2}\left[1 - \cos \omega_c (t-s) e^{-\nu (t-s)} \right]  - \frac{\omega_c}{\omega_c^2 + \nu^2} e^{-\nu(t-s)} \sin \omega_c (t-s)  \\ 
a_{12}(t-s) & = \frac{\omega_c}{\omega_c^2 + \nu^2}\left[1 - \cos \omega_c (t-s) e^{-\nu (t-s)} \right]  + \frac{\nu}{\omega_c^2 + \nu^2} e^{-\nu(t-s)} \sin \omega_c (t-s). 
\end{align}
Applying this to \eqref{eq:SAexp} gives
\begin{align}
S(t,\tau;k,\eta_{CT}(t)) = \exp\left[-\nu \int_\tau^t \abs{k_{\perp}}^2 (a_{11}^2 + a_{21}^2)  + \abs{k_3}^2 \left(\frac{1 - e^{-\nu (t-s)}}{\nu}\right)^2 ds \right]. \label{eq:Sttau} 
\end{align}
The contribution to \eqref{eq:Sttau} involving $k_3$ is computed in \cite{B17} for the unmagnetized case, and is given by 
\begin{align}
\exp\left[-\nu \int_\tau^t\abs{k_3}^2 \left(\frac{1 - e^{-\nu (t-s)}}{\nu}\right)^2 ds \right] & \\ & \hspace{-2cm} =  \exp\left[-\nu^{-1} k_3^2 \left( t-\tau + 2 \frac{e^{\nu(\tau-t)} - 1}{\nu} - \frac{e^{2\nu (\tau-t)}- 1 }{2\nu}\right)\right]. 
\end{align}
To compute the contribution in \eqref{eq:Sttau} from the magnetization, begin by expanding 
\begin{align}
(a_{11}^2 + a_{21}^2)(t-s) = \frac{1}{\nu^2 + \omega_c^2} \left[1 - \cos \omega_c (t-s) e^{-\nu (t-s)} \right]^2 + \frac{1}{\nu^2 + \omega_c^2} e^{-2\nu(t-s)} \sin^2 \omega_c (t-s). 
\end{align}
Hence, we have to express the following quantity as a function of $t-\tau$: 
\begin{align}
Q(t,\tau) :=  \int_\tau^t \left[1 - \cos \omega_c (t-s) e^{-\nu (t-s)} \right]^2 + e^{-2\nu(t-s)} \sin^2 \omega_c (t-s) ds. 
\end{align}
Expanding the square and using the Pythagorean identity gives 
\begin{align}
Q(t,\tau) & =  \int_\tau^t 1 - 2 \cos \omega_c (t-s) e^{-\nu (t-s)} + e^{-2\nu(t-s)} ds \\
& = (t-\tau) + \frac{1}{2\nu}\left(1 - e^{-2\nu(t-\tau)} \right) - \int_\tau^t e^{(-\nu + i\omega_c)(t-s)} + e^{(-\nu - i\omega_c)(t-s)} ds. 
\end{align}
Moreover, 
\begin{align}
\int_\tau^t e^{(-\nu + i\omega_c)(t-s)} + e^{(-\nu - i\omega_c)(t-s)} ds & = \frac{\nu + i\omega_c}{\nu^2 + \omega_c^2} \left(1 - e^{(-\nu + i\omega_c)(t-\tau)}\right)  + \frac{\nu - i \omega_c}{\nu^2  +\omega_c^2} \left(1 - e^{(-\nu - i\omega_c)(t-\tau)}\right) \\ 
&  \hspace{-3cm} = \frac{2\nu}{\nu^2 + \omega^2_c} - \frac{2\nu}{\nu^2 + \omega_c^2} e^{-\nu (t-\tau)} \cos \omega_{c}(t-\tau) + \frac{2\omega_c}{\nu^2 + \omega_c^2} \sin\omega_c(t-\tau) e^{-\nu (t-\tau)}. 
\end{align}
Applying this to \eqref{eq:Sttau} finally gives \eqref{def:S}. We similarly derive \eqref{eq:barEta}; the proof is omitted for the sake of brevity. 
\end{proof} 

From \eqref{eq:barEta} we further derive 
\begin{subequations}
\begin{align} \label{eq:ketabarCT}
k \cdot \bar{\eta}(\tau;k,\eta_{CT}(t)) & = a_{11} \abs{k_{\perp}}^2 + \frac{k_3^2}{\nu}\left( 1 - e^{-\nu(t-\tau)}\right) \\
\abs{\bar{\eta}(\tau;k,\eta_{CT}(t))}^2  
& = \frac{\abs{k_{\perp}}^2}{\nu^2 + \omega_c^2} \left( 1 - 2 \cos \omega_c (t-\tau) e^{-\nu (t-\tau)} + e^{-2\nu(t-\tau)}\right) \nonumber \\ & \quad + k_3^2 \left(\frac{1 - e^{-\nu (t-\tau)}}{\nu}\right)^2. 
\end{align}
\end{subequations}
Putting \eqref{eq:ketabarCT} together with Lemma \ref{lem:aijTech} and \eqref{eq:rhoetaCT} gives the following lemma. 
\begin{lemma} 
Define 
\begin{align}
\rho_{0;\nu} (t,k) & = S(t,k) \widehat{f_{in}}(k,\eta_{CT}(t,k)), 
\end{align}
and 
\begin{align}
K^{\nu}(t,k) & = S(t,k) \frac{1}{\abs{k}^2} \left[ \abs{k_{\perp}}^2\left[\frac{\nu}{\omega_c^2 + \nu^2}\left[1 - \cos \omega_c t e^{-\nu t} \right]  - \frac{\omega_c}{\omega_c^2 + \nu^2} e^{-\nu t} \sin \omega_c t \right] +  \frac{k_3^2}{\nu}\left( 1 - e^{-\nu t}\right) \right] \nonumber \\ 
& \quad \times \exp\left[-\pi \left( \frac{\abs{k_{\perp}}^2}{\nu^2 + \omega_c^2} \left( 1 - 2 \cos \omega_c t e^{-\nu t} + e^{-2\nu t}\right) + k_3^2 \left(\frac{1 - e^{-\nu t}}{\nu}\right)^2\right)\right].  \label{def:Knu}
\end{align}
Then, the density $\rho$ is given by the Volterra equation 
\begin{align}
\hat{\rho}(t,k) & =  \rho_{0;\nu}(t,k) + \int_0^t \hat{\rho}(\tau,k)K^\nu(t-\tau,k) d\tau. \label{eq:colVolt}
\end{align} 
\end{lemma} 
\subsubsection{Uniform Landau damping and enhanced collisions}
In this section we analyze \eqref{eq:colVolt} and complete the proof of \eqref{ineq:LDcol} (and hence Theorem \ref{thm:Coll}). 
The basic approach is similar to the unmagnetized case \cite{B17} (shift the Laplace transform by the expected exponential decay rate and use an approximation argument from the collisionless case to deduce that the dispersion relation is uniformly bounded away from zero near the imaginary axis) however it is significantly more complicated here. 

We will need the following Lemma. 
\begin{lemma}[Properties of $S$] \label{lem:propS}
	If $k_3 \neq 0$, then the following holds for all $t \geq 0$ and $\nu$ sufficiently small (depending only on universal constants): 
		 $S(t,k)$ is strictly decreasing and there exists a number $\delta_0> $ such that
		\begin{align}
		0 < S(t,k) < \exp\left(-\delta_0 \min(\nu k_3^2 t^3,\nu^{-1}k_3^2 t)\right) \exp\left(-\delta_0 \nu \abs{k_{\perp}}^2 \min(t^3,t) \right). \label{ineq:Sdecay}
		\end{align}
\end{lemma}
\begin{proof}
	The factor involving $k_3$ was estimated previously in \cite{B17}.
	To see the estimate on the $k_{\perp}$ factor, recall from above 
	\begin{align}
	S(t,k) = \exp\left[- \frac{\nu}{\nu^2 + \omega_c^2} \abs{k_{\perp}}^2 Q(t) \right] \exp\left[-\nu \int_0^t \abs{k_3}^2 \left(\frac{1 - e^{-\nu (t-s)}}{\nu}\right)^2 ds \right], 
	\end{align}
	where
	\begin{align}
	Q(t) :=  \int_0^t \left[1 - \cos \omega_c (t-s) e^{-\nu (t-s)} \right]^2 + e^{-2\nu(t-s)} \sin^2 \omega_c (t-s) ds.
	\end{align}
	For $t < \frac{1}{2\omega_c}$, (and $\nu$ sufficiently small relative to $\omega^{-1}_c$) note that 
	\begin{align}
	 \left[1 - \cos \omega_c (t-s) e^{-\nu (t-s)} \right]^2 + e^{-2\nu(t-s)} \sin^2 \omega_c (t-s) \gtrsim_{\omega_c} \abs{t-s}^2, 
	\end{align}
	hence, 
	\begin{align}
	Q(t) \gtrsim \int_0^t \abs{t-s}^2 ds \approx t^3. 
	\end{align}
	For $\frac{1}{2\omega_c} \leq t$, let $\delta$ be chosen small and divide $[0,t] = I_{G,\delta} \cup I_{B,\delta}$, with 
	\begin{align}
	I_{B,\delta} & = \set{s \in [0,t]: \left[1 - \cos \omega_c (t-s) e^{-\nu (t-s)} \right]^2 + e^{-2\nu(t-s)} \sin^2 \omega_c (t-s) < \delta } \\ 
	I_{G,\delta} & = [0,t] \setminus I_{B,\delta}. 
	\end{align}
	Notice that $\abs{I_{G,\delta}} \gtrsim_{\delta,\omega_c} t$,  
	It follows that 
	\begin{align}
	Q(t) \gtrsim \delta \abs{I_{G,\delta}} \gtrsim_{\delta,\omega_{c}} t. 
	\end{align}
	The lemma hence follows. 
\end{proof}
Let $\delta_0 > \delta > 0$ to be chosen small later and define 
\begin{align}
\phi(t,k) & = e^{\delta \nu^{1/3}  t} \hat{\rho} \\ 
K^{\nu,\delta}(t,k) & =  e^{\delta \nu^{1/3} t} K^\nu(t) \\ 
\phi_{0} & = S(t,k)e^{\delta \nu^{1/3} t} \hat{\rho}_{0;\nu}(t,k) 
\end{align}
and we have the Volterra equation 
\begin{align}
\phi(t,k) = \phi_0(t,k) + \int_0^t K^{\nu,\delta}(t-\tau) \phi(\tau,k) d\tau. 
\end{align}
As in the collisionless case with $k_3 \neq 0$ and $\widetilde{f^0_3} = 0$, $\cL[K^{\nu,\delta}](z,k)$ is holomorphic over $\mathbb C$ and there exist a $C >0$ such that $\cL[\phi](z,k)$ 
and $\cL[\phi_0](z,k)$ are holomorphic for $\Re z > C$. 
Hence for $\Re z > C$, 
\begin{align}
\cL[\phi](z,k) = \cL[\phi_0](z,k) + \cL[K^{\nu,\delta}](z,k)\cL[\phi](z,k).  
\end{align}  
First, we observe that the collisions add an additional deacy to the uniform-in-$\nu$ Landau damping of the passive transport evolution. 
\begin{lemma} \label{lem:phi0ctrl}
The following holds uniformly in $\nu$ for $\sigma \geq 0$, $m > 2$: 
\begin{align}
\norm{\abs{\partial_z}^{1/2} \brak{\grad,t\partial_z}^{\sigma} \phi_0}_{L^2_t L^2_x} & \lesssim \norm{h_{in}}_{H^\sigma_{m}}. 
\end{align}
\end{lemma}
\begin{proof}
Follows as  \eqref{ineq:LDrho0} in the collisionless case after applying Lemma \ref{lem:propS}. 
\end{proof}

The next lemma is straightforward, and provides the regularity in $\omega$ necessary to obtain Landau damping in Sobolev regularity. 
The proof is omitted for brevity.  

\begin{lemma} \label{lem:KnuSobReg}
There holds the following uniformly in $\nu$ and $k_{\perp}$ for $j \leq \sigma$. 
\begin{align}
\sup_{k \in \Integers^3_\ast: k_3 \neq 0} \sup_{\lambda \in [0,\infty)} \abs{\partial_\omega^j \cL[K^{\nu,\delta}](\lambda + i\omega,k)} & \lesssim_{j} 1.
\end{align}
\end{lemma}

The next two lemmas show that any non-decaying behavior introduced for small $\nu$ must occur in a fixed window in $(z,k) \in \Complex \times \Integers^3$. 
\begin{lemma} \label{lem:ColCtrl1}
For $\nu$ sufficiently small (depending on $\omega_c$), there holds the uniform bound 
\begin{align}
\sup_{\Re z \geq 0}\abs{\cL[K^{\nu,\delta}](z,k)} \lesssim \frac{1}{\abs{k}}. 
\end{align}
\end{lemma}
\begin{proof}
First consider the case $t < \nu^{-1/2}$.
In order to obtain a decay estimate from \eqref{def:Knu} in $k$ that does not depend badly on $\nu$, it is clear that the term involving $\abs{k^{\perp}}^2 \sin \omega_c t$ is potentially the worst term. 
For this term, we must use the exponential factor coming from the collisionless contribution. 	
For this, first note that for $t > 0$ there holds
\begin{align}
1-2e^{-\nu t} \cos \omega_c t + e^{-2\nu t} = 2e^{-\nu t}\left(\cosh \nu t - \cos \omega_c t\right) > 0.  
\end{align}
We will also separate short times ($t \lesssim \omega_c^{-1}$) from others in order to isolate powers of $t$, with the intention of using the smoothing from $S(t,k)$.
By $e^{-x} \lesssim x^{-1/2}$, there holds (from \eqref{def:Knu})
\begin{align}
\abs{K^{\nu,\delta}(t,k)} & \lesssim  \frac{S(t,k)}{\abs{k}^2} \abs{k_{\perp}}^2 \left( \mathbf{1}_{t<\omega_c^{-1}} (\nu t^2 + t) + \mathbf{1}_{t > \omega_c^{-1}} \nu + \abs{\frac{e^{-\frac{1}{2}\nu t} \sin \omega_c t}{\left(\cosh \nu t - \cos \omega_c t\right)^{1/2}\abs{k_{\perp}} }} \right) \\ 
& \quad \quad \times \exp\left(-\frac{\abs{k_{\perp}}^2}{\nu^2 + \omega_c^2} \mathbf{1}_{t < (\omega_c)^{-1}} t^2 \right) e^{-k_3^2 t^2} + \frac{S(t,k)}{\abs{k}^2} k_3^2 t e^{-k_3^2 t^2}. 
\end{align}
Notice for $t < \nu^{-1/2}$, where $n$ is such that $t \in [2n\pi - \pi, 2n\pi + \pi)$,
\begin{align} 
\abs{\frac{e^{-\frac{1}{2}\nu t} \sin \omega_c t}{\left(\cosh \nu t - \cos \omega_c t\right)^{1/2} }} \lesssim \frac{\abs{t - 2n\pi} }{\left(\abs{t - 2n\pi}^2 + (\nu t)^2 \right)^{1/2}} \lesssim 1. 
\end{align} 
Therefore, it follows that for $t < \nu^{-1/2}$, using also Lemma \ref{lem:propS},
\begin{align}
\abs{K^{\nu,\delta}(t,k)} & \lesssim \frac{1}{\abs{k}}\left(1+ k_3^2 t) \right) e^{-k_3^2 t}.   
\end{align}
On the other hand, for $t \geq \nu^{-1/2}$, we can simply bound by
\begin{align}
\abs{K^{\nu,\delta}(t,k)} & \lesssim S(t,k)\left(1 + k_3\right). 
\end{align}
Then, by Lemma \ref{lem:propS}, we have 
\begin{align}
\int_0^\infty \abs{K^{\nu,\delta}(t,k)} dt \lesssim \int_0^{\nu^{-1/2}} \abs{K^{\nu,\delta}(t,k)} dt  +\int_{\nu^{-1/2}}^\infty \abs{K^{\nu,\delta}(t,k)} dt \lesssim \frac{1}{\abs{k}}. 
\end{align}
\end{proof}

The next lemma is a little more involved but is based on similar principles. 

\begin{lemma} \label{lem:ColCtrl2}
Uniformly in $\nu$ sufficiently small, there holds 
\begin{align}
\sup_{k \in \Integers_\ast^3: k_3 \neq 0} \sup_{\Re z \geq 0}\abs{\cL[K^{\nu,\delta}](z,k)} \lesssim \frac{1}{\abs{\omega}}. 
\end{align}
\end{lemma}
\begin{proof}
As in the collisionless case (Lemma \ref{lem:clessImz}), it suffices to obtain uniform (in $k$, $z$, and $\nu$) absolute integrability of $\abs{\partial_t K^{\nu,\delta}(t,k)}$. 
Consider $\partial_t K^{\nu,\delta}$ (note the first two terms come from $\partial_t S$), 
\begin{align*}
\abs{\partial_t K^{\nu,\delta}} & \lesssim \nu^{-1} k_3^2 \abs{ 1 - 2 e^{-\nu t} + e^{- 2\nu t}} \abs{K^{\nu,\delta}(t,k)} \\ 
& \quad +  \abs{1 - \frac{2\nu}{\nu^2 + \omega_c^2} e^{-\nu t} \left(\nu \cos \omega_{c}t + \omega_c \sin \omega_c t\right)  - \frac{2\omega_c}{\nu^2 + \omega_c^2}  e^{-\nu t} \left(-\nu \sin\omega_c t + \omega_c \cos \omega_c t\right) + e^{-2 \nu t}} \\ 
& \quad\quad \times  \nu \abs{k_{\perp}}^2 \abs{K^\nu(t,k)} \\ 
& \quad + \left(\abs{k_{\perp}}^2\abs{2\omega_c \sin \omega_c t e^{-\nu t} + 2\nu \cos \omega_c t e^{-\nu t} -2\nu e^{-2\nu t}  } + k_3^2 e^{-\nu t} \left(\frac{1-e^{-\nu t}}{\nu}\right)\right) \abs{K^\nu(t,k)} \\ 
& \quad + \frac{1}{\abs{k}^2} \left[\left[ \abs{k_{\perp}}^2 \nu e^{-\nu t} + e^{-\nu t} \right] +  k_3^2 e^{-\nu t}\right] \\ 
& \quad\quad \times S(t,k)\exp\left[-\pi \left( \frac{\abs{k_{\perp}}^2}{\nu^2 + \omega_c^2} \left( 1 - 2 \cos \omega_c t e^{-\nu t} + e^{-2\nu t}\right) + k_3^2 \left(\frac{1 - e^{-\nu t}}{\nu}\right)^2\right)\right] \\ 
& := \cK_1 + \cK_2 + \cK_3 + \cK_4. 
\end{align*}
By Lemma \ref{lem:propS}, for $0 < t < \nu^{-1/2}$ we have  
\begin{align}
\cK_1 & \lesssim \nu k_3^2 t^2 \exp\left(-\delta_0 \min(\nu k_3^2 t^3,\nu^{-1}k_3^2 t)\right) \left(\nu + \frac{1}{\abs{k}^2} + k_3^2 t\right) \exp(-k_3^2 t^2), 
\end{align}
and hence (using Lemma \ref{lem:propS} for $t > \nu^{-1/2}$ as above ), 
\begin{align}
\int_0^\infty \cK_1(t,k) dt \lesssim 1. 
\end{align}
Then, by Lemma \ref{lem:propS}, for $0 < t < \nu^{-1/2}$ we have 
\begin{align}
\cK_2(t,k) & = \nu \abs{k_{\perp}}^2 \abs{1 - 2 \cos \omega_c t e^{-\nu t} + e^{-2 \nu t}} \abs{K^\nu(t,k)} \\ 
& \lesssim \nu \abs{k_{\perp}}^2 \left( \mathbf{1}_{t < \omega_c/2} t^2 + \mathbf{1}_{t > \omega_c/2}\right) \abs{K^\nu(t,k)} \\ 
& \lesssim \exp\left[-\delta_0 \nu \abs{k_{\perp}}^2 \min(t^3,t) \right] \nu \abs{k_{\perp}}^2 \left( \mathbf{1}_{t < \omega_c/2} t^2 + \mathbf{1}_{t > \omega_c/2}\right) \\ 
& \quad \times \frac{1}{\abs{k}^2} \left(\left(  (\nu \abs{k_{\perp}}^2 t + 1)\mathbf{1}_{t < \omega_c/2} t  +  \mathbf{1}_{t > \omega_c/2}\right) + 1 +  k_3^2 t \right) e^{- k_3^2 t^2}  \\ 
& \lesssim \left( 1+ k_3^2 t\right) e^{- k_3^2 t^2}, 
\end{align}
Hence, (using Lemma \ref{lem:propS} as above) 
\begin{align}
\int_0^\infty \cK_2(t,k) dt & \lesssim \int_0^{\nu^{-1/2}} \cK_2(t,k) dt + O(\frac{\nu}{\abs{k}^2}) \lesssim 1. 
\end{align}
Turn next to $\cK_3$. 
Here, for $t < \nu^{-1/2}$, by Lemma \ref{lem:propS},  
\begin{align*}
\cK_3 & \lesssim \abs{k_{\perp}}^2\left( \mathbf{1}_{t < \omega_c/2} t + \mathbf{1}_{t > \omega_c/2}\right) \frac{S(t,k)}{\abs{k}^2} \left(\abs{k_{\perp}}^2 \left( \mathbf{1}_{t < \omega_c/2}( \nu t^2 + t) + \mathbf{1}_{t > \omega_c/2}\right) + k_3^2 t\right) e^{-k_3^2 t^2} + k_3^2 t e^{-k_3^2 t^2} \\ 
& \lesssim (1 + k_3^2 t) e^{-k_3^2 t^2}.  
\end{align*}
Hence, using Lemma \ref{lem:propS} as above for $t > \nu^{-1/2}$, we have 
\begin{align} 
\int_0^\infty \cK_3 (t,k) dt & \lesssim 1. 
\end{align} 
The last term, $\cK_4$ is similar but easier and is hence omitted for brevity. 
\end{proof}

\begin{proof}[\textbf{Proof of Theorem \ref{thm:Coll} (ii)}]
Inequality \eqref{ineq:LDcol} will follow from Lemmas \ref{lem:KnuSobReg} and \ref{lem:phi0ctrl} provided we prove that there exists $\nu_0 > 0$ such that the following holds (where $\kappa$ is given in Lemma \ref{lem:MdEst}): 
\begin{align}
\inf_{\nu \in (0,\nu_0)} \inf_{k \in \Integers_\ast^3: k_3 \neq 0} \inf_{\Re z \geq 0} \abs{1-\cL[K^{\nu,\delta}](z,k)} \geq \kappa/2. \label{ineq:tildeKDisp}  
\end{align}	
Lemmas \ref{lem:ColCtrl1} and \ref{lem:ColCtrl2} show that there exists an $M > 0$ such that for all $\nu_0$ sufficiently small, 
\begin{align}
\inf_{\nu \in (0,\nu_0)} \inf_{k: k_3 \neq 0 \& \abs{k} \geq M} \inf_{z: \Re z \geq 0 \, \& \, \abs{\Im z} \geq M} \abs{1-\cL[K^{\nu,\delta}](z,k)} \geq \kappa/2. \label{ineq:tildeKDispCmpt}  
\end{align}	
Next, we will obtain \eqref{ineq:tildeKDisp} on $\abs{\Im z} \leq M$ and $\abs{k} \leq M$ by approximation from the collisionless case. 
Ultimately, we will apply the dominated  convergence theorem. 
By Lemma \ref{lem:propS} and calculations analogous to those in Lemma \ref{lem:ColCtrl1}, 
\begin{align}
\abs{K^{\nu,\delta}(t,k)} &\lesssim \brak{\frac{1-e^{-\nu t}}{\nu}} \exp\left(-\delta_0 \min(\nu k_3^2 t^3,\nu^{-1}k_3^2 t)\right) \\ & \quad\quad \times \exp\left[-\pi \left( \frac{\abs{k_{\perp}}^2}{\nu^2 + \omega_c^2} \left( 1 - 2 \cos \omega_c t e^{-\nu t} + e^{-2\nu t}\right) + k_3^2 \left(\frac{1 - e^{-\nu t}}{\nu}\right)^2\right)\right] \\ 
& \lesssim \frac{1}{\abs{k_3}} \exp\left(-\delta_0 \min(\nu k_3^2 t^3,\nu^{-1}k_3^2 t)\right) \exp\left[-k_3^2 \left(\frac{1 - e^{-\nu t}}{\nu}\right)^2 \right] \\ 
& \lesssim \mathbf{1}_{ \nu^{1/2}t < 1} \frac{1}{\abs{k_3}} \exp\left[-\frac{1}{2} k_3^2 t^2 \right] + \mathbf{1}_{\nu^{1/2} t > 1} \frac{1}{\abs{k_3}} \exp\left(-\delta_0 \min(\nu k_3^2 t^3,\nu^{-1}k_3^2 t)\right).  \label{ineq:Knbds}
\end{align}
Consider the following integral for $\lambda \geq 0$; we have for $\nu$ sufficiently small: 
\begin{align}
\cL[K^{\nu,\delta}](z,k) & = \int_0^\infty e^{-(\lambda + i\omega)t} K^{\nu,\delta}(t,k) dt  \\
& =  \int_0^{\nu^{-1/2}} e^{-(\lambda + i\omega)t} K^{\nu,\delta}(t,k) dt + O\left(\frac{1}{k_3^2}\exp[-k_3^2 \nu^{-1/12}]\right). 
\end{align}
By \eqref{ineq:Knbds}, we can then apply the dominated convergence theorem to the first term and deduce for all $z,k$ with $\Re z \geq 0$,
$\lim_{\nu \rightarrow 0} \cL[K^{\nu,\delta}](z,k) = \cL[K^0](z,k)$ pointwise. 
Moreover, using \eqref{ineq:Knbds}, it is clear that we have uniform (in $\nu$, $k$, and $z$) bounds on the derivatives
\begin{align}
\sup_{z \in \Complex: \Im z \leq M\, \& \,0 \leq \Re z \leq M} \abs{\partial_z^j \cL[K^{\nu,\delta}](z,k)} \lesssim_j 1. 
\end{align} 
Using the derivative estimates and the restrictions  $\abs{\Im z} \leq M$, $\abs{k} \leq M$, $0 \leq \Re z \leq M$ (it is clear that we can add this additional restriction by the definition of the Laplace transform) imply we can upgrade the pointwise convergence to uniform convergence. 
That is, for all $\eps > 0$, there exists $\nu_\eps$ such that 
\begin{align}
\sup_{\nu \in (0,\nu_\eps)}\sup_{k\in \Integers_\ast^3 : k_3 \neq 0 \& \abs{k} \leq M} \sup_{z \in \Complex: \Re z \geq 0 \, \& \, \abs{z} \leq M} \abs{\cL[K^\nu](z,k) - \cL[K^0](z,k)} \leq \eps. 
\end{align} 
The estimate \eqref{ineq:tildeKDisp} then follows from Lemma \ref{lem:MdEst}. As discussed above, this completes the proof of \eqref{ineq:LDcol}, and hence, Theorem \ref{thm:Coll}.
\end{proof}

\appendix
\section{Identities and estimates for the generalized Bessel functions} \label{sec:Bessel}
Recall the generalized Bessel functions
\begin{equation}
\label{}
I_\alpha(x)=i^{-\alpha}J_{\alpha}(ix)=\sum_{m=0}^\infty \frac{1}{m!\Gamma(m+\alpha+1)}\left(\frac{x}{2}\right)^{2m+\alpha}
\end{equation}
where $\alpha\in\mathbb R$ and $x\in\mathbb C$.
Theses  functions enjoy the following identities:
\begin{align}
&\exp{(z\cos\theta)}=I_0(z)+2\sum_1^\infty I_n(z)\cos n\theta
\label{Besid0}\\
&I_{n-1}(a)-I_{n+1}(a)=\frac{2n}{a}I_n(a)
\label{Besid1}
\end{align}
which is used crucially in the proof.  
From \eqref{Besid0}, we have
	\begin{equation}
	\label{e:a}
	\text{e}^a=\sum_{n=-\infty}^{\infty}I_n(a)=I_0(a)+2\sum_1^\infty I_n(a).
	\end{equation}
	Note that since each $I_{n}(a)$ is positive, it is straightforward to get
	\begin{equation}
	\label{}
	I_{n}(a) \le \text{e}^a
	\end{equation}
	for any $n\in\mathbb{N}$ and $a\in\mathbb{R}$.
By summing up \eqref{Besid1} in $n$, we get
\begin{equation}
\label{1st:mom}
I_{0}(a)+I_{1}(a)=\sum_{n=1}^\infty (I_{n-1}(a)-I_{n+1}(a))=\sum_{n=1}^\infty \frac{2n}{a}I_{n}(a).
\end{equation}
For $I_{0}(a)$ and $I_{1}(a)$, we have the following bounds.
\begin{lemma}
	\label{I0I1}
	The following inequalities hold
	\begin{align}
	\label{}
	&I_{0}(a)
	\lesssim
	\frac{1}{\sqrt{a}}\text{e}^{a};
	\label{bd:I0}
	\\
	&I_{1}(a)
	\lesssim
	\frac{1}{\sqrt{a}}\text{e}^{a}.
	\label{bd:I1}
	\end{align}
\end{lemma}
\begin{proof}
	By the definition of $I_{n}(a)$, we obtain
	\begin{align}
	I_{0}(a)
	=
	\sum_{m=0}^\infty \frac{1}{m!^2}\left(\frac{a}{2}\right)^{2m}
	\le
	\text{e}^{a/2}\sup_{m\ge0} \frac{1}{m!}\left(\frac{a}{2}\right)^{m}.
	\end{align}
	Denote $m_0 = [a/2]$ and observe that the sequence $\frac{1}{m!}\left(\frac{a}{2}\right)^{m}$ is increasing for $m \leq m_0$ and decreasing for $m \geq m_0$ and hence, 
	\begin{align}
	\label{}
	\sup_m \frac{1}{m!}\left(\frac{a}{2}\right)^{m}
	&=
	\frac{1}{m_0!}\left(\frac{a}{2}\right)^{m_0}
	\lesssim
	\frac{1}{2\pi \sqrt{m_0}}\left(\frac{e}{m_0}\right)^{m_0}\left(\frac{a}{2}\right)^{m_0}
	\nonumber\\&
	\lesssim 
	\frac{1}{2\pi \sqrt{m_0}}\text{e}^{a/2}\left(\frac{a}{2m_0}\right)^{m_0}
	\lesssim
	\frac{1}{2\pi \sqrt{m_0}}\text{e}^{a/2}.
	\end{align}
	Therefore, we may bound $I_{0}(a)$ as
	\begin{equation}
	\label{}
	I_{0}(a)
	\lesssim
	\frac{1}{2\pi \sqrt{m_0}}\text{e}^{a}
	\end{equation}
	from where \eqref{bd:I0} follows.
	Similar argument gives
	\begin{equation}
	\label{}
	I_{1}(a)
	\lesssim
	\frac{1}{\sqrt{a}}\text{e}^{a}
	\end{equation}
	completing the proof of the lemma.
\end{proof}

\bibliographystyle{abbrv}
\bibliography{eulereqns}

\end{document}